\documentclass[12pt]{amsart}
\usepackage{amsfonts,amssymb,latexsym,amsmath,verbatim,url,color,mathrsfs,tikz,extarrows,enumerate,amscd,verbatim}
\usepackage[margin=1in]{geometry}
\usepackage[pdfpagelabels,bookmarksopen=true,hidelinks=true]{hyperref}
\usepackage[nameinlink]{cleveref}
\usetikzlibrary{matrix,arrows}
    \DeclareFontFamily{U}{wncy}{}
    \DeclareFontShape{U}{wncy}{m}{n}{<->wncyr10}{}
    \DeclareSymbolFont{mcy}{U}{wncy}{m}{n}
    \DeclareMathSymbol{\Sh}{\mathord}{mcy}{"58} 
\def\N{\mathbf{N}}
\def\Z{\mathbf{Z}}

\def\C{\mathbf{C}}
\def\F{\mathbf{F}}
\def\A{\mathbf{A}}
\def\P{\mathbf{P}}
\def\G{\mathbf{G}}
\def\H{\mathrm{H}}
\usepackage[cal=boondoxo, scr=boondoxo]{mathalfa}

\usepackage{baskervald}

\DeclareMathOperator{\Hom}{Hom}
\DeclareMathOperator{\End}{End}
\DeclareMathOperator{\Spec}{Spec}
\DeclareMathOperator{\Sch}{Sch}
\DeclareMathOperator{\Grp}{Grp}
\DeclareMathOperator{\op}{op}
\DeclareMathOperator{\GL}{GL}
\DeclareMathOperator{\PGL}{PGL}
\DeclareMathOperator{\fppf}{fppf}
\DeclareMathOperator{\Br}{Br}

\DeclareMathOperator{\LPBr}{LPBr}
\DeclareMathOperator{\LFBr}{LFBr}
\DeclareMathOperator{\Ext}{Ext}
\DeclareMathOperator{\Pic}{Pic}
\DeclareMathOperator{\im}{im}
\DeclareMathOperator{\Mat}{Mat}
\DeclareMathOperator{\HomS}{\mathcal{H\hspace{-0.3ex}o\hspace{-0.3ex}m}}
\DeclareMathOperator{\EndS}{\mathcal{E\hspace{-0.1ex}n\hspace{-0.3ex}d}}
\DeclareMathOperator{\ExtS}{\mathcal{E\hspace{-0.1ex}x\hspace{-0.2ex}t}}
\DeclareMathOperator{\colim}{colim}
\DeclareMathOperator{\Aut}{Aut}
\DeclareMathOperator{\id}{id}
\DeclareMathOperator{\coker}{coker}
\DeclareMathOperator{\Sym}{Sym}
\DeclareMathOperator{\rk}{rank}
\DeclareMathOperator{\Set}{Set}

\DeclareMathOperator{\Proj}{Proj}
\newcommand{\mf}[1]{\mathfrak{#1}}
\newcommand{\ms}[1]{\mathscr{#1}}
\newcommand{\mc}[1]{\mathcal{#1}}
\newcommand{\mb}[1]{\mathbf{#1}}
\newcommand{\ml}[1]{\mathsf{#1}}
\newcommand{\mr}[1]{\mathrm{#1}}
\newcommand{\DEF}[1]{\textcolor{blue}{#1}}

\newcommand{\et}{\operatorname{\acute et}}
\newcommand{\BM}[1]{\begin{bmatrix} #1 \end{bmatrix}}
\newtheoremstyle{mystyle}
  {}
  {}
  {}
  {0pt}
  {}
  {.}
  { }
  {\textbf{\thmname{#1}\thmnumber{ #2}}\thmnote{ (#3)}}

\theoremstyle{theorem}
\numberwithin{equation}{subsubsection}
\newtheorem{theorem}[subsubsection]{Theorem}
\newtheorem{corollary}[subsubsection]{Corollary}
\newtheorem{proposition}[subsubsection]{Proposition}
\newtheorem{lemma}[subsubsection]{Lemma}

\theoremstyle{definition}
\newtheorem{definition}[subsubsection]{Definition}
\newtheorem{remark}[subsubsection]{Remark}
\newtheorem{question}[subsubsection]{Question}
\newtheorem{example}[subsubsection]{Example}

\newtheorem{principle}[subsubsection]{Principle}

\newtheoremstyle{pgstyle} {} {} {} {} {} {.} { } {\textbf{\thmname{#1}\thmnumber{#2}}\thmnote{ (#3)}}
\theoremstyle{pgstyle}
\newtheorem{pg}[subsubsection]{}

\newtheoremstyle{spgstyle} {} {} {} {} {} {.} { } {\textit{\thmname{#1}\thmnumber{#2}}\thmnote{ (#3)}}
\theoremstyle{spgstyle}
\newtheorem{spg}{}[subsubsection]

\begin{document}
\title{Locally free twisted sheaves of infinite rank}
\author{Aise Johan de Jong}
\author{Max Lieblich}
\author{Minseon Shin}
\begin{abstract} 
We study twisted vector bundles of infinite rank on gerbes, giving a new spin on Grothendieck's famous problem on the equality of the Brauer group and cohomological Brauer group. We show that the relaxed version of the question has an affirmative answer in many, but not all, cases, including for any algebraic space with the resolution property and any algebraic space obtained by pinching two closed subschemes of a projective scheme. We also discuss some possible theories of infinite rank Azumaya algebras, consider a new class of ``very positive'' infinite rank vector bundles on projective varieties, and show that an infinite rank vector bundle on a curve in a surface can be lifted to the surface away from finitely many points. 
\end{abstract}
\date{\today}
\maketitle
\setcounter{tocdepth}{2}
\tableofcontents

\section{Introduction} \label{sec01}

One of the fundamental invariants of a scheme is its Brauer group. By analogy with the Brauer group of a field, which classifies central simple algebras up to Morita equivalence, the \emph{Brauer group} $\Br(X)$ of a scheme $X$ classifies equivalence classes of \emph{Azumaya algebras}, which are \'etale twists of matrix algebras $\Mat_{n \times n}(\mc{O}_X)$. Just as with fields, any Azumaya algebra $\mc{A}$ has an associated \'etale cohomology class $[\mc{A}]$ in the \emph{cohomological Brauer group} $\Br'(X) := \H_{\et}^2(X,\G_m)_{\text{\rm tors}}$. This defines a functorial embedding \[ \Br(X) \subseteq \Br'(X) \] called the \emph{Brauer map}. 

This paper originated with the problem of determining whether the Brauer map is an isomorphism, which is one of the central problems at the interface of modern algebra and algebraic geometry. Whereas the Brauer map is an isomorphism when $X$ is the spectrum of a field (i.e. every torsion \'etale cohomology class over a field is represented by some central division algebra), for arbitrary schemes this is not necessarily true. By a theorem of Gabber \cite{DEJONG-AROG2003}, the Brauer map is known to be an isomorphism when $X$ admits an ample line bundle, but the question of surjectivity of the Brauer map remains open more generally, e.g. for schemes admitting an ample family of line bundles (which include smooth separated schemes over a field). 

Using the language of \emph{twisted sheaves} \cite{LIEBLICH-MOTS2007}, the $\Br = \Br'$ question may be rephrased as follows. By Giraud's theory of non-abelian cohomology \cite{GIRAUD-CNA1971}, the cohomological Brauer group $\Br'(X)$ classifies $\G_m$-gerbes $\ms X\to X$ of finite order. For a $\G_{m}$-gerbe $\ms{X}$ over a scheme $X$, the condition that the class $[\ms{X}]$ comes from an Azumaya algebra of rank $r^{2}$ is equivalent to saying that $\ms{X}$ admits a twisted vector bundle of rank $r$. The existence of such vector bundles has strong consequences: for example, if there is a twisted vector bundle of rank $r$ on $\ms{X}$, then $[\ms{X}]$ is $r$-torsion in $\H_{\et}^{2}(X,\G_{m})$.

If $[\ms{X}]$ is arbitrary (not necessarily torsion), we can ask whether $\ms{X}$ admits twisted locally free $\mc{O}_{\ms{X}}$-modules of \emph{infinite\/} rank. Let \[ \LPBr(X) \subseteq \H_{\et}^{2}(X,\G_{m}) \] be the submonoid consisting of classes $[\ms{X}]$ such that $\ms{X}$ admits a twisted locally free $\mc{O}_{\ms{X}}$-module (not necessarily of finite rank); this is roughly equivalent to saying that $[\ms{X}]$ comes from an Azumaya algebra of possibly infinite rank. Our main theorem asserts the existence of twisted vector bundles of infinite rank on all $\G_{m}$-gerbes over a large class of algebraic spaces (in greater generality than the known answers to the $\Br = \Br'$ question).

\begin{theorem} \label{MAINTHM} For an algebraic space $X$, the inclusion $\LPBr(X) \subseteq \H_{\et}^{2}(X,\G_{m})$ is an equality if at least one of the following hold. \begin{enumerate} \item[(i)] $X$ has the resolution property. \item[(ii)] $X$ admits an ample family of line bundles. \item[(iii)] $X$ is a separated algebraic surface. \item[(iv)] $X$ is a quotient $Y/G$ over some base scheme $S$, where $Y$ is an $S$-algebraic space satisfying at least one of the conditions (i) -- (iii), and $G$ is an $S$-group scheme acting freely on $Y$ and such that $G \to S$ is an affine, flat, finitely presented morphism with geometrically irreducible fibers. \end{enumerate} \end{theorem} 

At first glance, it seems vector bundles of infinite rank would be more difficult to understand than those of finite rank. In fact, infinite rank vector bundles have a rather simple local structure (often by arguments involving the Eilenberg swindle), enough to suggest the following as the key guiding principle of our main results:

\begin{principle}\label{main principle} 
  Vector bundles of infinite rank
  have simpler structure than vector bundles of finite rank.
\end{principle}

The prime example of this is Bass's theorem on the triviality of projective modules of countably infinite rank over Noetherian rings. \begin{theorem}[Bass] \label{BASSTHM} \cite{BASS-BPMAF1963} Let $A$ be a ring such that $A/J(A)$ is Noetherian, where $J(A)$ is the Jacobson radical. If $M$ is a projective $A$-module of countably infinite rank, then $M$ is free. \end{theorem} 

As we will see, local uniqueness theorems such as these yield global existence results that are quite strong in comparison to their finite-rank analogues (which are often false). 

\begin{remark}
One can also read Bass's theorem as telling us that the various naïve notions of ``infinite rank vector bundle'' -- that is, locally free sheaves or simply locally projective sheaves -- coincide.
\end{remark}

To prove \Cref{MAINTHM}, we first consider the affine case, for which we prove a twisted version of Bass's theorem. Then we show that in each case $X$ admits a cover $\pi : X' \to X$ where $X'$ is affine and $\pi$ is a surjective \emph{affine-pure\/} morphism, which is roughly characterized by the property that the pushforward of a locally free module by such a morphism remains locally free (in case (ii), Jouanolou's trick provides an example of such a map). The cases (i), (ii), (iii) are special cases of (iv), which reduces to the claim that $G \to S$ is affine-pure under the stated assumptions. This is enough to conclude since the $\LPBr(X) = \H_{\et}^{2}(X,\G_{m})$ property descends under affine-pure morphisms.

For non-affine schemes, it is not true that all infinite rank vector bundles are trivial (in fact, we give examples of infinite rank vector bundles on $\P^{1}$ that are not decomposable as direct sums of line bundles). In response to this difficulty, we introduce a class of infinite rank vector bundles which we call \emph{very positive vector bundles}. We show that, for projective schemes over an infinite field, very positive vector bundles exist and are unique up to isomorphism; we view this as another example of \Cref{main principle}. This uniqueness forces them to descend through simple pushouts, allowing us to prove that $\LPBr(X) = \H_{\et}^{2}(X,\G_{m})$ for a class of schemes not included in \Cref{MAINTHM}.

\begin{theorem} Let $X$ be the colimit of a diagram $Z\rightrightarrows Y$ where $Z$ and $Y$ are smooth and projective over an infinite field and the two arrows are closed immersions with disjoint images. Then $\LPBr(X) = \H_{\et}^{2}(X,\G_{m})$. \end{theorem} 

In our investigations, one essential subtlety lies in what precisely one means by ``vector bundle of infinite rank''. The bundles we consider here are essentially the discrete parts of the types of infinite vector bundles described by Drinfeld in \cite{DRINFELD-IDVBIAG2006}. A consequence of this restriction is that we lose dualizability; this is why the subset $\LPBr(X)$ of $\H_{\et}^2(X,\G_m)$ that we study is only a submonoid. It is an interesting question to consider what we would happen if one allows the compact duals as in Drinfeld's theory; then one seems to lose tensor products, so one would no longer have a submonoid, but merely a subset closed under inversion!

The $\Br = \Br'$ and $\LPBr(X) = \H_{\et}^{2}(X,\G_{m})$ problems are two instances of a broader desire in algebraic geometry to find appropriate geometric or ring-theoretic objects corresponding to cohomology classes. For $\H^{2}_{\et}(X,\G_{m})$, there is an earlier, different approach due to Taylor \cite{TAYLOR-ABBG1982} who defined the \emph{bigger Brauer group} $\widetilde{\Br}(X)$. These are the classes representable by certain quasi-coherent $\mc{O}_{X}$-algebras that are neither unital nor locally free in general. By work of Heinloth--Schr{\"o}er \cite{HEINLOTHSCHROER-TBBGATS2009}, it is known that $\widetilde{\Br}(X) = \H_{\et}^{2}(X,\G_{m})$ for Noetherian algebraic stacks $X$ whose diagonal is quasi-affine; they show that all $\G_{m}$-gerbes over such $X$ admit a twisted coherent sheaf that locally contains an invertible summand.

\begin{pg}[Outline] We provide an outline of the paper. In \Cref{sec02} we discuss twisted sheaves and define the submonoid $\LPBr(X)$ over an arbitrary locally ringed site. We define infinite rank Azumaya algebras (\Cref{sec02-01}) and Brauer-Severi varieties (\Cref{sec02-03}) and discuss their relationship to twisted vector bundles. After giving some sufficient conditions for a morphism to be affine-pure (\Cref{sec03-01}), we prove the main theorem in \Cref{sec03-02}. We modify the example of a scheme for which $\Br \ne \Br'$ (due to Edidin--Hassett--Kresch--Vistoli \cite{EDIDINHASSETTKRESCHVISTOLI-BGAQS2001}) to produce a scheme for which $\LPBr(X) = 0$ but $\H_{\et}^{2}(X,\G_{m}) = \Z$. In \Cref{sec04}, we introduce very positive vector bundles and prove that the proper schemes admitting a very positive vector bundle are exactly the projective ones (\Cref{sec04-02}) and that very positive vector bundles are unique up to isomorphism (\Cref{sec04-03}); we also prove that every finite rank vector bundle over a projective scheme admits a finite-length ``forward resolution'' by the very positive vector bundle (\Cref{sec04-04}). In \Cref{sec05}, we prove that infinite-dimensional invertible matrices lift under any surjective ring map, which implies in particular that all infinite rank vector bundles lift from a curve to an ambient surface away from finitely many points (again supporting \Cref{main principle}). In \Cref{sec06}, we prove infinite rank analogues of some results used in the definition of the Brauer map, including the Skolem-Noether theorem (generalizing a result of Courtemanche--Dugas \cite{COURTEMANCHEDUGAS-AOTEAOAFM2016}) and the fact that endomorphism algebras of projective modules are central algebras. \end{pg}

\begin{pg}[Acknowledgements] We are grateful to Jarod Alper, Giovanni Inchiostro, and Will Sawin for helpful conversations. \end{pg}

\section{Twisted sheaves and Azumaya algebras of possibly infinite rank} \label{sec02}

\subsection{Terminology about modules} In this paper we will use the following terminology regarding modules (which may be different from existing definitions).

\begin{definition} \label{0027} Let $\mc{C}$ be a locally ringed site, let $\mc{F}$ be an $\mc{O}_{\mc{C}}$-module. \par We say that $\mc{F}$ is \DEF{locally free} if for every object $U$ of $\mc{C}$ there is a covering $\{U_i \to U\}_{i \in I}$ such that for every $i$ the restriction $\mc{F}|_{U_i}$ is a free $\mc{O}_{U_i}$-module.
\par We say that $\mc{F}$ is \DEF{locally projective}
if for every object $U$ of $\mc{C}$ there is a covering
$\{U_i \to U\}_{i \in I}$ such that for every $i$ the restriction $\mc{F}|_{U_i}$ is a
direct summand of a free $\mc{O}_{U_i}$-module.
\par We say a locally projective $\mc{O}_{\mc{C}}$-module $\mc{F}$ has \DEF{positive rank} if for every object $U$ of $\mc{C}$ there is a covering $\{U_i \to U\}_{i \in I}$ such that for every $i \in I$ there is a surjection
$\mc{F}|_{U_i} \to \mc{O}_{U_i}$.
\par We say a locally projective $\mc{O}_{\mc{C}}$-module $\mc{F}$ is \DEF{countably generated}
for every object $U$ of $\mc{C}$ there is a covering $\{U_i \to U\}_{i \in I}$ such that
for every $i \in I$ there is a countable set $I_{i}$ and a surjection $\mc{O}_{U_{i}}^{\oplus I_{i}} \to \mc{F}|_{U_i}$ of $\mc{O}_{U_i}$-modules. \end{definition}

\begin{pg} We note that every locally projective $\mc{O}_{\mc{C}}$-module $\mc{F}$ is quasi-coherent. Indeed, after localizing on $\mc{C}$ if necessary, we may assume that $\mc{O}_{\mc{C}}^{\oplus I} \simeq \mc{F} \oplus \mc{G}$ for some $\mc{O}_{\mc{C}}$-module $\mc{G}$; then $\mc{F}$ is the cokernel of the composition $\mc{O}_{\mc{C}}^{\oplus I} \to \mc{G} \to \mc{O}_{\mc{C}}^{\oplus I}$. \end{pg}

\begin{remark} For an algebraic space $X$, we are primarily interested in classes in the \'etale cohomology group $\H_{\et}^{2}(X,\G_{m})$. Thus whenever we discuss modules, cohomology, and gerbes on an algebraic space $X$, we will assume without further mention that the underlying site $\mc{C}$ is the small \'etale site $X_{\et}$ of $X$. (Recall that we have $\H_{\et}^{2}(X,\G_{m}) \simeq \H_{\fppf}^{2}(X,\G_{m})$ by \cite[(11.7) (1)]{Gro68c}.) \end{remark}

\begin{pg} For a scheme $X$ and a quasi-coherent $\mc{O}_{X}$-module $\mc{F}$, the following \Cref{0006} shows that $\mc{F}$ is locally projective for the fpqc topology if and only if it is locally projective for the Zariski topology. \end{pg}

\begin{theorem}[Raynaud-Gruson] \label{0006} \cite{RAYNAUDGRUSON-CDPEDP1971}, \cite[05A9]{SP} Let $A$ be a ring, let $M$ be an $A$-module, let $A \to B$ be a faithfully flat ring map. If $M \otimes_{A} B$ is a projective $B$-module, then $M$ is a projective $A$-module. \end{theorem}

\begin{question}[Infinite-rank Hilbert Theorem 90] Let $X$ be a scheme, let $\mc{E}$ be a quasi-coherent $\mc{O}_{X}$-module. If there exists an fppf (resp. \'etale) cover $X' \to X$ such that $\mc{E}|_{X'}$ is a free $\mc{O}_{X'}$-module, then does there exist an \'etale (resp. Zariski) cover $X'' \to X$ such that $\mc{E}|_{X''}$ is a free $\mc{O}_{X''}$-module? In other words, are the maps \[ \H_{\mr{Zar}}^{1}(X,\GL_{I}) \to \H_{\et}^{1}(X,\GL_{I}) \to \H_{\fppf}^{1}(X,\GL_{I}) \] isomorphisms for any index set $I$? By \Cref{BASSTHM}, we know that if $X$ is Noetherian and $I$ is countable, then the answer is ``yes''. See also \cite[05VF]{SP}. \end{question}

\subsection{Twisted sheaves} \label{sec02-02} We briefly recall some background about gerbes and twisted sheaves; see \cite[3.1.1]{LIEBLICH-TSAPIP2008} for a reference.

\begin{pg} Let $\mc{C}$ be a locally ringed site, let $\alpha \in \H^{2}(\mc{C},\G_{m})$ be a class, and let $\ms{X} \to \mc{C}$ be a $\G_{m}$-gerbe with $[\ms{X}] = \alpha$. An $\mc{O}_{\ms{X}}$-module $\ms{F}$ is said to be \DEF{1-twisted} if the natural inertial action $\G_{m} \times \mc{F} \to \mc{F}$ equals the action obtained by restricting the module action $\mc{O}_{\ms{X}} \times \mc{F} \to \mc{F}$ \cite[3.1.1.1]{LIEBLICH-TSAPIP2008}. \end{pg}

\begin{pg} Let $f : Y \to X$ be a morphism of algebraic spaces and let $\ms{X} \to X$ be a $\G_{m}$-gerbe. If $\mc{G}$ is a 1-twisted $\mc{O}_{Y \times_{X} \mc{X}}$-module, then the pushforward $f_{\ast}\mc{G}$ is a 1-twisted $\mc{O}_{\mc{X}}$-module. If $\mc{F}$ is a 1-twisted $\mc{O}_{\mc{X}}$-module, then the pullback $f^{\ast}\mc{F}$ is a 1-twisted $\mc{O}_{Y \times_{X} \mc{X}}$-module. \end{pg}

\begin{pg} We may equivalently describe a twisted sheaf using a hypercovering $K$ on which $\alpha$ is representable.
If the underlying site $\mc{C}$ has fibre products and (enough) disjoint unions, then we can assume our
hypercovering $K$ is just a simplicial object $U_{\bullet}$ of $\mc{C}$, see \cite[0DB1]{SP}.
Then $\alpha$ is given by an element $\alpha_{2} \in \Gamma(U_{2},\G_{m})$ satisfying
the cocycle condition on $U_{3}$, and an \DEF{$\alpha$-twisted} module is a pair
$(\mc{F}_{0}, \varphi_{1})$ where $\mc{F}_{0}$ is an $\mc{O}_{U_{0}}$-module and $\varphi_{1} : p_{1}^{\ast}\mc{F}_{0} \to p_{2}^{\ast}\mc{F}_{0}$ is an isomorphism of $\mc{O}_{U_{1}}$-modules satisfying the condition $p_{23}^{\ast}\varphi_{1} \circ p_{12}^{\ast}\varphi_{1} = \alpha_{2} \cdot p_{13}^{\ast}\varphi_{1}$ on $U_{2}$. Letting the hypercovering vary, we obtain an equivalence of categories $\mc{F} \mapsto (\mc{F}_{0},\varphi_{1})$ between these two viewpoints, and for each of the properties $\mr{P}$ in \Cref{0027}, we have that $\mc{F}$ satisfies $\mr{P}$ if and only if $\mc{F}_{0}$ satisfies $\mr{P}$. \end{pg}

\begin{definition} \label{0004} For a locally ringed site $\mc{C}$, we define \DEF{$\LPBr(\mc{C})$} (resp. \DEF{$\LFBr(\mc{C})$}) to be the subset of classes $\alpha \in \H^{2}(\mc{C},\G_{m})$ such that any $\G_{m}$-gerbe $\ms{X} \to \mc{C}$ with $[\ms{X}] = \alpha$ admits a countably generated locally projective (resp. locally free) 1-twisted $\mc{O}_{\ms{X}}$-module of positive rank. \end{definition}

\begin{lemma} \label{0002} For a locally ringed site $\mc{C}$, the subsets $\LPBr(\mc{C})$ and $\LFBr(\mc{C})$ of $\H^{2}(\mc{C},\G_{m})$ are additive submonoids. \end{lemma} \begin{proof} The structure sheaf $\mc{O}_{\mc{C}}$ is a 0-twisted line bundle, hence $0 \in \LFBr(\mc{C})$. Let $\alpha_{1},\alpha_{2} \in \H^{2}(\mc{C},\G_{m})$ be classes; we may assume that both $\alpha_{1},\alpha_{2}$ are representable on the hypercovering $U_{\bullet}$. If $(\mc{F}_{i,0},\varphi_{i,1})$ is an $\alpha_{i}$-twisted sheaf, then $(\mc{F}_{1,0} \otimes_{\mc{O}_{U_{0}}} \mc{F}_{2,0} , \varphi_{1,1} \otimes \varphi_{2,1})$ is an $(\alpha_{1} + \alpha_{2})$-twisted sheaf. If $\mc{F}_{1,0},\mc{F}_{2,0}$ are locally projective (resp. locally free, resp. of positive rank), then so is $\mc{F}_{1,0} \otimes_{\mc{O}_{U_{0}}} \mc{F}_{2,0}$. \end{proof}

\begin{remark} Suppose $(\mc{F}_{0},\varphi_{1})$ is a locally projective $\alpha$-twisted $\mc{O}_{\mc{C}}$-module. If $\mc{F}_{0}$ has finite rank, then $(\HomS_{\mc{O}_{\mc{C}}}(\mc{F}_{0},\mc{O}_{\mc{C}}),\varphi_{1}^{\vee})$ is a locally projective $(-\alpha)$-twisted $\mc{O}_{\mc{C}}$-module. However, if $\mc{F}_{0}$ does not have finite rank, then $\HomS_{\mc{O}_{\mc{C}}}(\mc{F}_{0},\mc{O}_{\mc{C}})$ need not be quasi-coherent (in particular, not locally projective). \end{remark}

\begin{question} \label{0012} Does there exist an algebraic space $X$ such that $\LPBr(X)$ is not a subgroup of $\H_{\et}^{2}(X,\G_{m})$ (i.e. does not contain additive inverses)? \end{question}

\begin{question} \label{0025}\label{infinite brauer problem} For which algebraic spaces $X$ is the inclusion $\LPBr(X) \to \H_{\et}^{2}(X,\G_{m})$ an equality? \end{question}

\begin{pg} \label{0005} For any algebraic space $X$, let $\Br(X)$ (resp. $\Br'(X)$) be the Brauer group (resp. cohomological Brauer group) of $X$. We have inclusions \begin{center} \begin{tikzpicture}[>=angle 90] 
\matrix[matrix of math nodes,row sep=2em, column sep=2em, text height=2ex, text depth=0.5ex] { 
|[name=11]| \Br(X) & |[name=12]|  & |[name=13]| \Br'(X) \\ 
|[name=21]| \LFBr(X) & |[name=22]| \LPBr(X) & |[name=23]| \H^{2}_{\et}(X,\G_{m}) \\
}; 
\draw[->,font=\scriptsize] (11) edge (13) (21) edge (22) (22) edge (23) (11) edge (21) (13) edge (23); \end{tikzpicture} \end{center} of additive submonoids of $\H^{2}_{\et}(X,\G_{m})$. The classes in $\Br(X)$ correspond to $\G_{m}$-gerbes $\ms{X}$ admitting a 1-twisted locally free $\mc{O}_{\ms{X}}$-module of finite rank. By a theorem of Gabber \cite{DEJONG-AROG2003}, the inclusion $\Br(X) \subseteq \Br'(X)$ is known to be an equality if $X$ is a quasi-compact scheme admitting an ample line bundle. \end{pg}

\begin{lemma} \label{LemmaB} If $X$ is a locally Noetherian algebraic space, we have $\LFBr(X) = \LPBr(X)$. \end{lemma} \begin{proof} Let $\ms{X} \to X$ be a $\G_{m}$-gerbe admitting a countably generated locally projective 1-twisted $\mc{O}_{\ms{X}}$-module $\mc{E}$ of positive rank. There is a locally Noetherian scheme $X'$ with an \'etale surjection $X' \to X$ such that $\ms{X}_{X'}$ is trivial; for any section $s : X' \to \ms{X}_{X'}$, the restriction $s^{\ast}(\mc{E}|_{\ms{X}_{X'}})$ is a locally projective $\mc{O}_{X'}$-module of positive rank. For an affine open subscheme $U = \Spec A$ of $X'$, the $A$-module $\Gamma(U,\mc{E}|_{U})$ is projective by \Cref{0006}. If $\Gamma(U,\mc{E}|_{U})$ has finite rank, then $\mc{E}|_{U}$ is Zariski-locally free; otherwise, since $A$ is Noetherian, $\Gamma(U,\mc{E}|_{U})$ is free by Bass' theorem \cite{BASS-BPMAF1963}. \end{proof}

\subsection{Azumaya algebras} \label{sec02-01}

\begin{definition}[Azumaya algebras] \label{20180329-aahg} Let $\mc{C}$ be a locally ringed site. Given an $\mc{O}_{\mc{C}}$-algebra $\mc{A}$, a \DEF{trivialization} of $\mc{A}$ is a pair \[ (\mc{E},\varphi) \] where $\mc{E}$ is a locally free $\mc{O}_{\mc{C}}$-module of positive rank, and $\varphi : \mc{A} \to \EndS_{\mc{O}_{\mc{C}}}(\mc{E})$ is an $\mc{O}_{U}$-algebra isomorphism. We say that $\mc{A}$ is an \DEF{Azumaya $\mc{O}_{\mc{C}}$-algebra} if for every object $U \in \mc{C}$ there exists a covering $\{U_{i} \to U\}_{i \in I}$ such that for each $i \in I$ the restriction $\mc{A}|_{U_{i}}$ admits a trivialization. \end{definition}

\begin{remark} According to our definition, an Azumaya $\mc{O}_{X}$-algebra $\mc{A}$ is not quasi-coherent unless it has finite rank. \end{remark}

\begin{definition}[The Brauer map] \label{20180329-aahl} Let $\mc{A}$ be an Azumaya $\mc{O}_{\mc{C}}$-algebra. We define the \DEF{gerbe of trivializations} of $\mc{A}$ to be the category $\mc{G}_{\mc{A}}$ whose objects are tuples $(U,\mc{E},\varphi)$ where $U \in \mc{C}$ is an object and $(\mc{E},\varphi)$ is a trivialization of $\mc{A}|_{U}$; a morphism $(U_{1},\mc{E}_{1},\varphi_{1}) \to (U_{2},\mc{E}_{2},\varphi_{2})$ is a pair $(f,\rho)$ where $f : U_{1} \to U_{2}$ is a morphism in $\mc{C}$ and $\rho : f^{\ast}\mc{E}_{2} \to \mc{E}_{1}$ is an $\mc{O}_{U_{1}}$-module isomorphism such that the diagram \begin{center} \begin{tikzpicture}[>=angle 90] 
\matrix[matrix of math nodes,row sep=3em, column sep=2em, text height=1.7ex, text depth=0.5ex] { 
|[name=11]| \mc{A}|_{U_{1}} & |[name=12]| f^{\ast}\mc{A}|_{U_{2}} \\ 
|[name=21]| \EndS_{\mc{O}_{U_{1}}}(\mc{E}_{1}) & |[name=22]| \EndS_{\mc{O}_{U_{1}}}(f^{\ast}\mc{E}_{2}) \\
}; 
\draw[-,font=\scriptsize,transform canvas={yshift= 1pt}](11) edge (12);
\draw[-,font=\scriptsize,transform canvas={yshift=-1pt}](11) edge (12);
\draw[->,font=\scriptsize] (21) edge node[below=0pt] {$c_{\rho}$} (22) (11) edge node[left=0pt] {$\varphi_{1}$} (21) (12) edge node[right=0pt] {$f^{\ast}\varphi_{2}$} (22); \end{tikzpicture} \end{center} commutes, where $c_{\rho}$ is induced by conjugation-by-$\rho$. Then $\mc{G}_{\mc{A}}$ is a stack fibered in groupoids over $\mc{C}$, and it is a $\G_{m}$-gerbe by \Cref{20180329-aaga} (Skolem-Noether) and \Cref{20210804-12}. The association $\mc{A} \mapsto \mc{G}_{\mc{A}}$ defines an extension of the usual Brauer map. \end{definition}

\begin{proposition} \label{20180329-aahf} Let $X$ be an algebraic space, let $\alpha \in \H_{\et}^{2}(X,\G_{m})$ be a class. The following conditions are equivalent. \begin{enumerate} \item[(i)] We have $\alpha = [\mc{G}_{\mc{A}}]$ for some Azumaya $\mc{O}_{X}$-algebra $\mc{A}$. \item[(ii)] The class $\alpha$ is in $\LFBr(X)$. \end{enumerate} \end{proposition} \begin{proof} Let $\mc{G}$ be the $\G_{m}$-gerbe corresponding to $\alpha$. \par (i)$\Rightarrow$(ii): Let $\mc{A}$ be an Azumaya algebra on $X$ and suppose that $\mc{G}$ is the gerbe of trivializations of $\mc{A}$. The assignment $(U,\mc{E},\varphi) \mapsto (U,\mc{E})$ and $(f,\rho) \mapsto (f,\rho)$ defines a 1-twisted locally free $\mc{O}_{\mc{G}}$-module. \par (ii)$\Rightarrow$(i): Suppose there exists a 1-twisted locally free $\mc{O}_{\mc{G}}$-module $\mc{E}$. The endomorphism sheaf $\EndS_{\mc{O}_{\mc{G}}}(\mc{E})$ is a $0$-twisted $\mc{O}_{\mc{G}}$-algebra, hence is isomorphic to the pullback of some $\mc{O}_{X}$-algebra $\mc{A}$. There exists an \'etale cover $X' \to X$ such that $\mc{G}' := X' \times_{X} \mc{G}$ admits a 1-twisted line bundle $\mc{L}'$. If we denote by $\mc{E}'$ the $\mc{O}_{X'}$-module whose pullback to $\mc{G}'$ is $\mc{E}|_{\mc{G}'} \otimes_{\mc{O}_{\mc{G}'}} \mc{L}'^{-1}$, then we have an $\mc{O}_{X'}$-algebra isomorphism $\mc{A}|_{X'} \simeq \EndS_{\mc{O}_{X'}}(\mc{E}')$, hence $\mc{A}$ is an Azumaya algebra. \end{proof}

\subsection{Brauer-Severi varieties} \label{sec02-03}

Let $I$ be an index set, let $\Z^{\oplus I}$ be a free abelian group with basis indexed by $I$, and define \DEF{projective $I$-space} to be $\DEF{\P_{\Spec \Z}^{I}} := \Proj_{\Spec \Z} \Sym_{\Z}^{\bullet} \Z^{\oplus I}$. As in Bass' theorem, from which we know that infinite-rank vector bundles on finite-type schemes are simple, finite-rank vector bundles on infinite-dimensional varieties are also simple: by work of Barth--van de Ven \cite{BARTHVANDEVEN-ADCFA2BOPS1974}, Tyurin \cite{TYURIN-FDVBOIV1976}, Sato \cite{SATO-OTDOIEVBOPSAGV1977}, it is known that finite rank vector bundles on $\P_{\C}^{\N}$ are direct sums of line bundles.

We define a Brauer-Severi variety to be an \'etale twist of $\P^{I}$ for some $I$:

\begin{definition} \label{20180329-aaht} Let $X \to S$ be a morphism of schemes. We say that $X$ is a \DEF{Brauer-Severi scheme} over $S$ if there exists an \'etale surjection $S' \to S$, an index set $I$, and an isomorphism $X \times_{S} S' \simeq \P_{S'}^{I}$ of $S'$-schemes. \end{definition}

\begin{pg} \label{20180329-aahu} Define $\GL_{I} : \Sch^{\op} \to \Grp$ to be the sheaf of groups sending $S \mapsto \Aut_{\mc{O}_{S}\text{-mod}}(\mc{O}_{S}^{\oplus I})$ and let \DEF{$\PGL_{I}$} be the quotient $\GL_{I}/\G_{m}$. A module automorphism of $\mc{O}_{S}^{\oplus I}$ induces a graded algebra automorphism of $\Sym_{\mc{O}_{S}}^{\bullet} \mc{O}_{S}^{\oplus I}$, hence a scheme automorphism of $\P_{S}^{I}$; thus we obtain a natural homomorphism \begin{align} \label{20180329-aahu-01} \PGL_{I} \to \Aut_{\text{sch}}(\P^{I}) \end{align} of sheaves of groups on $\Sch$. (Note that neither $\GL_{I}$ nor $\PGL_{I}$ are representable if $I$ is infinite.) \end{pg}

\begin{lemma} \label{20180329-aahv} The map \labelcref{20180329-aahu-01} is an isomorphism. \end{lemma} \begin{proof} Let $S$ be a scheme and let $\varphi : \P_{S}^{I} \to \P_{S}^{I}$ be an $S$-scheme automorphism. By \cite[II, (4.2.3)]{EGA}, a morphism $T \to \P_{\Spec \Z}^{I}$ is given by an invertible $\mc{O}_{T}$-module $\mc{L}$ and an $I$-indexed collection of global sections $\sigma_{i} \in \Gamma(T,\mc{L})$ which are globally generating, i.e. $T = \bigcup_{i \in I} T_{\sigma_{i}}$. By \Cref{20180329-aahw}, after taking a Zariski cover of $S$, we may assume that $S = \Spec A$ is affine and that $\varphi$ is the morphism corresponding to the line bundle $\mc{O}_{\P_{S}^{I}}(n)$ and globally generating sections $\{\sigma_{i}\}_{i \in I}$. Since $\varphi$ induces an automorphism of $\Pic(\P_{S}^{I})$, we have $n = 1$ (note that $\Gamma(\P_{S}^{I},\mc{O}_{\P_{S}^{I}}(-1)) = 0$). Let $\{\sigma_{i}\}_{i \in I}$ and $\{\tau_{i}\}_{i \in I}$ be the $I$-indexed collections of sections corresponding to $\varphi$ and $\varphi^{-1}$, respectively. For any $I$-indexed collection of sections $v_{i} \in \Gamma(\P_{S}^{I},\mc{O}_{\P_{S}^{I}}(1))$, let $\BM{\{v_{i}\}_{i \in I}}$ denote the $I \times 1$ column vector whose $i$th entry is $v_{i}$; then there exist matrices $\ml{M},\ml{N} \in \Mat_{I \times I}(A)$ such that $\BM{\{\sigma_{i}\}_{i \in I}} = \ml{M} \BM{\{x_{i}\}_{i \in I}}$ and $\BM{\{\tau_{i}\}_{i \in I}} = \ml{N} \BM{\{x_{i}\}_{i \in I}}$. Since $\ml{N}\ml{M}\BM{\{x_{i}\}_{i \in I}}$ corresponds to the identity morphism, by \cite[II, (4.2.3)]{EGA} there exists a unit $u \in A^{\times}$ such that $\ml{N}\ml{M} = u \cdot \ml{id}_{I}$, thus $\ml{M} \in \GL_{I}(A)$. \end{proof}

\begin{pg} In \Cref{20180329-aahw}, we compute the Picard group of infinite-dimensional projective spaces over an arbitrary scheme (the case $S = \Spec k$ for a field $k$ was proved by Tyurin \cite[Proposition 1.1 (1)]{TYURIN-FDVBOIV1976}). \end{pg}

\begin{lemma} \label{20180329-aahw} For any scheme $S$, the map \begin{align} \label{20180329-aahw-eqn-01} \Pic(S) \oplus \Gamma(S,\underline{\Z}) \to \Pic(\P_{S}^{I}) \end{align} sending $(\mc{L},n) \mapsto \mc{L} \otimes \mc{O}_{\P_{S}^{I}}(n)$ is surjective. \end{lemma} \begin{proof} Let $\mc{L}$ be a line bundle on $\P_{S}^{I}$. Let $\Lambda$ denote the set of finite subsets of $I$, so that $I$ is the filtered union $I = \varinjlim_{\lambda \in \Lambda} \lambda$. For each $\lambda \in \Lambda$, we have a projection map $\Z^{\oplus I} \to \Z^{\oplus \lambda}$ which induces a closed immersion $\P_{S}^{\lambda} \to \P_{S}^{I}$. We know that \labelcref{20180329-aahw-eqn-01} is an isomorphism when $I$ is finite, thus there exist invertible $\ms{O}_{S}$-modules $\mc{M}_{\lambda}$ and integers $n_{\lambda}$ such that \[ \mc{L}|_{\P_{S}^{\lambda}} \simeq \mc{M}_{\lambda}|_{\P_{S}^{\lambda}} \otimes_{\mc{O}_{\P_{S}^{\lambda}}} \mc{O}_{\P_{S}^{\lambda}}(n_{\lambda}) \] on $\P_{S}^{\lambda}$. Since we have linear transition maps $\P_{S}^{\lambda} \to \P_{S}^{\lambda'}$, we have in fact $n_{\lambda} = n_{\lambda'}$ and $\mc{M}_{\lambda} \simeq \mc{M}_{\lambda'}$ for any two $\lambda,\lambda' \in \Lambda$; thus, after taking a Zariski cover of $S$ and replacing $\mc{L}$ by $\mc{L} \otimes_{\mc{O}_{\P_{S}^{I}}} (\mc{M}^{-1}|_{\P_{S}^{I}} \otimes_{\mc{O}_{\P_{S}^{I}}} \mc{O}_{\P_{S}^{I}}(-n))$ for this common invertible $\mc{O}_{S}$-module $\mc{M}$ and integer $n$, we may assume that in fact $\mc{L}|_{\P_{S}^{\lambda}}$ is trivial for all $\lambda$, and that $S = \Spec A$ is affine. Let $U_{i} := \mr{D}_{+}(x_{i})$ denote the distinguished affine open subscheme of $\P_{S}^{I}$ associated to $i \in I$ and set $U_{\lambda} := \bigcap_{i \in \lambda} U_{i}$ for all $\lambda$. By \Cref{20180329-aahx}, the restriction $\mc{L}|_{U_{i}}$ is trivial for all $i \in I$; then $\mc{L}$ is described by a collection of units $u_{i_{1},i_{2}} \in \Gamma(U_{\{i_{1},i_{2}\}},\G_{m})$ for $i_{1},i_{2} \in I$, corresponding to the transition map between trivializations of $\mc{L}$ on $U_{\{i_{1},i_{2}\}}$. Fix two distinct indices $i_{1}^{\circ},i_{2}^{\circ} \in I$ and let $\lambda$ be a finite subset of $I$ containing the indices of all variables $x_{i}$ that appear in $u_{i_{1}^{\circ},i_{2}^{\circ}}$; then $\mc{L}|_{U_{\{i_{1}^{\circ},i_{2}^{\circ}\}}}$ is isomorphic to the pullback of a line bundle on $\P_{S}^{\lambda}$ via the projection map $\P_{S}^{I} \setminus \mr{V}_{+}(\{x_{i}\}_{i \in I \setminus \lambda}) \to \P_{S}^{\lambda}$. Since $\mc{L}|_{\P_{S}^{\lambda}}$ is trivial by the above, we have that $\mc{L}|_{U_{\{i_{1}^{\circ},i_{2}^{\circ}\}}}$ is trivial; hence we may assume that $u_{i_{1}^{\circ},i_{2}^{\circ}} = 1$. For all $i' \in I \setminus \{i_{1}^{\circ},i_{2}^{\circ}\}$, the sequence \[ 0 \to \Gamma(U_{i'},\mc{O}_{\P_{S}^{I}}) \to \Gamma(U_{\{i',i_{1}^{\circ}\}},\mc{O}_{\P_{S}^{I}}) \oplus \Gamma(U_{\{i',i_{2}^{\circ}\}},\mc{O}_{\P_{S}^{I}}) \to \Gamma(U_{\{i',i_{1}^{\circ},i_{2}^{\circ}\}},\mc{O}_{\P_{S}^{I}}) \] is exact, so we may assume that $u_{i',i_{1}^{\circ}} = u_{i',i_{2}^{\circ}} = 1$ for all $i' \in I$, and (by the same argument) that $u_{i,i'} = 1$ for all $i,i' \in I$. \end{proof}

\begin{lemma} \label{20180329-aahx} Let $A$ be a filtered colimit of subrings $A = \varinjlim_{\lambda \in \Lambda} A_{\lambda}$ such that, for all $\lambda$, the inclusion $A_{\lambda} \to A$ admits a (ring-theoretic) retraction $A \to A_{\lambda}$. Suppose $M$ is an invertible $A$-module such that $M \otimes_{A} A_{\lambda}$ is a trivial $A_{\lambda}$-module for all $\lambda$. Then $M$ is trivial. \end{lemma} \begin{proof} By \cite[05N7]{SP}, there exists some $\lambda \in \Lambda$ and a finitely presented $A_{\lambda}$-module $M_{\lambda}$ such that $M_{\lambda} \otimes_{A_{\lambda}} A \simeq M$. Tensoring by $- \otimes_{A} A_{\lambda}$ then implies that $M_{\lambda}$ is trivial, hence $M$ is trivial. \end{proof}

\begin{pg} Using \Cref{20180329-aahv}, we may extend the dictionary between Azumaya algebras and twisted vector bundles (\Cref{20180329-aahf}) to include Brauer-Severi varieties. \end{pg}

\begin{proposition} For any scheme $S$ and index set $I$, there is a bijective correspondence between isomorphism classes of Azumaya $\mc{O}_{S}$-algebras of rank $I$ and Brauer-Severi schemes of relative dimension $I$. \end{proposition} \begin{proof} Both are \'etale $\PGL_{I}$-torsors on $S$, by \Cref{20180329-aaga} and \Cref{20180329-aahv} respectively. \end{proof}

\section{Pushing forward twisted sheaves via affine-pure morphisms} \label{sec03}

\subsection{Affine-pure morphisms} \label{sec03-01}

\begin{definition} \label{0007} A morphism $f : X \to Y$ of algebraic spaces is said to be \DEF{affine-pure} if $f$ is affine and $f_{\ast}\mc{O}_{X}$ is a locally projective $\mc{O}_{Y}$-module. \end{definition}

\begin{lemma} \label{0009} \cite[3.5 $(\gamma)$]{BATTISTONROMAGNY-ROAGSOGR2008} Let $B \to A$ be a ring map such that $A$ is a projective $B$-module, and let $M$ be a projective $A$-module. Then $M$ is projective as a $B$-module. \end{lemma} \begin{proof} There exist a split $B$-linear surjection $B^{\oplus I} \to A$ and a split $A$-linear surjection $A^{\oplus J} \to M$, hence a split $B$-linear surjection $(B^{\oplus I})^{\oplus J} \to M$. \end{proof}

\begin{lemma} \label{0028} Let $A$ be a ring, let $M$ be a projective $A$-module, set $S := \Spec A$ and let $\mc{E}$ be the quasi-coherent $\mc{O}_{S}$-module corresponding to $M$. Then $\mc{E}$ has positive rank if and only if $M \otimes_{A} \kappa(\mf{p}) \ne 0$ for all primes $\mf{p}$ of $A$. \end{lemma} \begin{proof} The ``only if'' implication is clear. For the ``if'' implication, choose an index set $I$ and a $A$-linear surjection $\pi : A^{\oplus I} \to M$; then $\pi$ admits a $A$-linear section $\xi : M \to A^{\oplus I}$; denoting by $\{\ml{e}_{i}\}_{i \in I}$ the basis of $A^{\oplus I}$, set $\alpha_{i} := \xi(\pi(\ml{e}_{i})) = (\alpha_{i,j})_{i \in I}$, so that $\{\alpha_{i}\}_{i \in I}$ constitutes a set of generators for $M$ viewed as a $A$-submodule of $A^{\oplus I}$. Let $\mf{p}$ be a prime ideal of $A$; since $M \otimes_{A} \kappa(\mf{p})$ is the image of $\xi \pi \otimes \id_{\kappa(\mf{p})}$, by hypothesis there exists some $(i,j) \in I^{2}$ such that $\alpha_{i,j} \not\in \mf{p}$. After inverting $\alpha_{i,j}$, we may assume that $\alpha_{i,j}$ is a unit; then the composition $p_{i} \xi \pi : A^{\oplus I} \to A$ is surjective, where $p_{i}$ is the projection onto the $i$th coordinate of $A^{\oplus I}$; hence $p_{i} \xi : M \to A$ is surjective.  \end{proof}

\begin{lemma} \label{0008} Let $f : X \to Y$ be an affine-pure morphism of algebraic spaces, let $\mc{E}$ be a locally projective $\mc{O}_{X}$-module. Then $f_{\ast}\mc{E}$ is a locally projective $\mc{O}_{Y}$-module. If in addition $f$ is surjective and $\mc{E}$ has positive rank, then $f_{\ast}\mc{E}$ has positive rank. \end{lemma} \begin{proof} We may work \'etale-locally on $Y$ so that we may assume that $X = \Spec A$ and $Y = \Spec B$ are affine; then the first claim follows from \Cref{0006} and \Cref{0009}. For the second claim, we use \Cref{0028}. Let $\mf{p}$ be a prime of $A$ lying over the prime $\mf{q}$ of $B$. The surjection $A \otimes_{B} \kappa(\mf{q}) \to \kappa(\mf{p})$ induces a surjection $M \otimes_{B} \kappa(\mf{q}) \to M \otimes_{A} \kappa(\mf{p})$ of $A$-modules, so if $M \otimes_{A} \kappa(\mf{p}) \ne 0$ then $M \otimes_{B} \kappa(\mf{q}) \ne 0$ as well. \end{proof}

\begin{lemma} \label{LemmaCDE} Let $f : X \to Y$ be a morphism of algebraic spaces.
    \begin{enumerate}
        \item[(i)] For any morphism $Y' \to Y$, let $f' : X' \to Y'$ be the base change of $f$. If $f$ is affine-pure, then $f'$ is affine-pure.
        \item[(ii)] Suppose there is an fpqc cover $\{Y_{i} \to Y\}_{i \in I}$ such that the base change $f_{i} : X_{i} \to Y_{i}$ is affine-pure for each $i \in I$. Then $f$ is affine-pure.
        \item[(iii)] Let $g : Y \to Z$ be a morphism of algebraic spaces. If $f$ and $g$ are affine-pure, then $g \circ f$ is affine-pure.
    \end{enumerate}
\end{lemma}
\begin{proof} (i): The property ``affine'' is stable under arbitrary base change. The map $(f_{\ast}\mc{O}_{X})|_{Y'} \to (f')_{\ast}\mc{O}_{Y'}$ is an isomorphism since $f$ is affine. \par (ii): The property ``affine'' is fpqc local on the base. For each $i$, we have that $(f_{i})_{\ast}\mc{O}_{X_{i}}$ is a locally projective $\mc{O}_{Y_{i}}$-module. As above, the map $(f_{\ast}\mc{O}_{X})|_{Y_{i}} \to (f_{i})_{\ast}\mc{O}_{X_{i}}$ is an isomorphism, thus $f_{\ast}\mc{O}_{X}$ is a locally projective $\mc{O}_{Y}$-module by \Cref{0006}. \par (iii): The property ``affine'' is stable under composition. We have that $f_{\ast}\mc{O}_{X}$ is a locally projective $\mc{O}_{Y}$-module by assumption, so $g_{\ast}f_{\ast}\mc{O}_{X}$ is a locally projective $\mc{O}_{Z}$-module by \Cref{0008}. \end{proof}

\begin{pg} We will use the following lemma, which is an analogue of \cite[II, Lemma 4]{GABBER-STOAA1978} in the infinite rank case: \end{pg}

\begin{lemma} \label{LemmaG} Let $f : X \to Y$ be a surjective, affine-pure, finitely presented morphism. If $\alpha \in \H^2_{\et}(Y, \G_m)$
is an element such that $f^{\ast}\alpha$ is in $\LPBr(X)$, then $\alpha$ is in $\LPBr(Y)$. \end{lemma}
\begin{proof} Let $\ms{Y} \to Y$ be a $\G_{m}$-gerbe corresponding to $\alpha$, set $\ms{X} := \ms{Y} \times_{Y} X$ with projection map $\pi : \ms{X} \to \ms{Y}$. We have that $\pi$ is surjective, affine-pure (by \Cref{LemmaCDE} (i)), and finitely presented. Let $\mc{E}$ be a countably generated locally projective 1-twisted $\mc{O}_{\ms{X}}$-module of positive rank; then $\pi_{\ast}\mc{E}$ is a locally projective 1-twisted $\mc{O}_{\ms{Y}}$-module which is of positive rank by \Cref{0008} and is countably generated since $f$ is finitely presented. \end{proof}

\begin{lemma} \label{LemmaH} Let $S$ be an algebraic space and let $f : X \to Y$ be a morphism of algebraic spaces over $S$. If $f$ satisfies at least one of the following conditions, then $f$ is affine-pure. \begin{enumerate} 
\item[(a)] $f$ is affine, flat, of finite presentation and all geometric fibers are integral,
\item[(b)] $f$ is affine, flat, of finite presentation and all geometric fibers are Cohen-Macaulay and irreducible,
\item[(c)] $f$ is an fppf $G$-torsor for an $S$-group scheme $G$ such that $G \to S$ is an affine-pure morphism,
\item[(d)] $f$ is an fppf $G$-torsor for an $S$-group scheme $G$ such that $G \to S$ is an affine, flat, finitely presented morphism, all of whose geometric fibers are irreducible.
\end{enumerate} \end{lemma}
\begin{proof}
(a): This is a restatement of \cite[05FT]{SP}.

(b): By \Cref{LemmaCDE}(ii), we may assume that $Y = \Spec A$ for a henselian local ring $A$; let $y \in Y$ be the closed point. By \cite[05MD]{SP} it is enough to verify that $\mc{O}_{X}$ is pure along the special fiber $X_{y}$. By \cite[031Q]{SP}, the fibers of $f$ do not have embedded associated points; hence, since the fibers of $f$ are irreducible, the only relative assassins of $\mc{O}_{X}$ over $Y$ are the generic points of fibers. Let $y' \in Y$ be a point and let $x$ be the generic point of $X_{y}$. By going-down for flatness, there exists some $x' \in X_{y'}$ such that $x' \leadsto x$; we may replace $x'$ by the generic point of $X_{y'}$. This shows that the generic point of $X_{y'}$ specializes to the generic point of $X_{y}$, hence there are no impurities of $\mc{O}_{X}$ over $y$.

(c): By definition, there exists an fppf cover $Y' \to Y$ and an isomorphism of $Y'$-schemes $X|_{Y'} \simeq G_{Y'}$; thus $f$ is affine-pure by \Cref{LemmaCDE} (i) and (ii).

(d): By (c), it suffices to show that $G \to S$ is affine-pure. By \cite[Exp. $\text{VI}_{\text{A}}$, 1.1.1]{DEMAZUREGROTHENDIECK-SGA3}, the geometric fibers of $G \to S$ are Cohen-Macaulay, so this follows from (b). 
\end{proof}

\begin{lemma} \label{LemmaI} Let $X$ be an algebraic space having at least one of the following
properties: \begin{enumerate}
\item[(a)] $X$ has an ample family of invertible modules \cite[II, 2.2.4]{SGA6} \cite[0FXR]{SP},
\item[(b)] $X$ has an ample invertible module,
\item[(c)] $X$ is quasi-projective over an affine,
\item[(d)] $X$ is quasi-affine,
\item[(e)] $X$ is quasi-compact and quasi-separated and has the resolution property,
\item[(f)] $X$ is a separated algebraic surface,
\item[(g)] there exists a base scheme $S$, an $S$-algebraic space $Y$ satisfying at least one of the
conditions (a) -- (f), an $S$-group scheme $G$ acting freely on $Y$ and such that $G \to S$ is affine-pure and finitely presented, and $X$ is the quotient $Y/G$. \end{enumerate}
Then there exists an affine scheme $T$ and a surjective, affine-pure, finitely presented morphism $T \to X$. \end{lemma}
\begin{proof} In each case we use \Cref{LemmaH} to get the morphism to be affine-pure. \par (a): This follows from Thomason's extension of the Jouanolou trick \cite[Proposition 4.4]{WEIBEL-HAKT1989} which implies the existence of an affine scheme $T$ and a smooth surjection $T \to X$ such that there exists a Zariski cover $U \to X$ and an isomorphism $T \times_{X} U \simeq \A_{U}^{n}$ of $U$-schemes. The map $\A_{\Z}^{n} \to \Spec \Z$ is affine-pure so $T \to X$ is affine-pure by \Cref{LemmaCDE} (i), (ii). \par (b): This is a special case of (a). \par (c): This is a special case of (b). \par (d): This is a special case of (c). \par (e): By \cite[1.1]{GROSS-TGOSAS2013} we can write $X = X'/\GL_N$ for a quasi-affine scheme $X'$. Then $X' \to X$ is surjective and affine-pure by \Cref{LemmaH} (d). By case (d), we can find an affine scheme $T$ with a surjective, affine-pure morphism $T \to X'$; then \Cref{LemmaCDE} (iii) implies that $T \to X' \to X$ is affine-pure. \par (f): This follows from (e) by \cite[5.3]{GROSS-TRPOAS2012}. \par (g): This follows from \Cref{LemmaH}(d) and \Cref{LemmaCDE}(iii). \end{proof}

\begin{lemma} \label{LemmaJ} For any algebraic space $X$ as in \Cref{LemmaI} and $\alpha$ in $\H^2_{\et}(X, \G_m)$, there
exists an affine scheme $T$ and a surjective, affine-pure, finitely presented morphism $T \to X$ such that
$\alpha$ is Zariski-locally trivial on $T$. \end{lemma}
\begin{proof} By \Cref{LemmaI}, there exists an affine scheme $T'$ and a surjective affine-pure morphism $T' \to X$. By \cite[01FW]{SP} there is an \'etale surjection $T'' \to T'$ such that $\alpha|_{T''}$ is trivial; we refine the cover $T'' \to T'$ so that $T''$ is affine. Then \cite[02LH]{SP} implies that there exists a surjective, finite locally free $T \to T'$ which factors through $T'' \to T$ Zariski-locally on $T$, i.e. $\alpha|_{T}$ is Zariski-locally trivial. We have that surjective finite locally free morphisms are affine-pure, and the composite $T \to T' \to X$ is surjective affine-pure by \Cref{LemmaCDE} (iii). \end{proof}

\begin{question} Given an affine scheme $X$ and a class $\alpha \in \H_{\et}^{2}(X,\G_{m})$, does there exist a surjective, affine-pure morphism $T \to X$ such that $\alpha|_{T}$ is trivial? \end{question}

\subsection{Proof of the main theorem} \label{sec03-02}

In this section we prove \Cref{LemmaK} and give some applications and questions.

\begin{proposition} \label{LemmaKS5} Let $A$ be a Noetherian ring, set $X := \Spec A$, and let $\ms{X} \to X$ be a $\G_{m}$-gerbe. Let $\mc{F}$ be a countably generated locally projective 1-twisted $\mc{O}_{\ms{X}}$-module of positive rank. Then $\mc{F}$ is a projective object of the category of quasi-coherent 1-twisted $\mc{O}_{\ms{X}}$-modules. Moreover, every quasi-coherent 1-twisted $\mc{O}_{\ms{X}}$-module is a quotient of a direct sum of copies of $\mc{F}$. \end{proposition}
\begin{proof} For the first claim, it is enough to show that
\begin{align} \label{LemmaKS5-eqn-03} \Ext_{\mc{O}_{\ms{X}}}^1(\mc{F}, \mc{H}) = 0 \end{align}
for any quasi-coherent 1-twisted $\mc{O}_{\ms{X}}$-module $\mc{H}$. We will show
\begin{align} \label{LemmaKS5-eqn-01} \ExtS_{\mc{O}_{\ms{X}}}^1(\mc{F}, \mc{H}) = 0 \end{align}
and
\begin{align} \label{LemmaKS5-eqn-02} \H^1(X, \HomS_{\mc{O}_{\ms{X}}}(\mc{F}, \mc{H})) = 0 \end{align}
which implies \labelcref{LemmaKS5-eqn-03} by the local-to-global spectral sequence for Ext. \par
For \labelcref{LemmaKS5-eqn-01}, let $X'$ be a finitely presented affine $X$-scheme with an $X$-morphism $X' \to \ms{X}$. If $\mc{F}$ has finite rank, then we may replace $X'$ by a Zariski cover so that $\mc{F}|_{X'}$ is a free $\mc{O}_{X'}$-module of finite rank; then $\Ext_{\mc{O}_{X'}}^{1}(\mc{F}|_{X'},\mc{H}|_{X'}) = 0$. If $\mc{F}$ has countably infinite rank, then $\mc{F}|_{X'}$ is in fact a free $\mc{O}_{X'}$-module by Bass' theorem, so $\Ext_{\mc{O}_{X'}}^{1}(\mc{F}|_{X'},\mc{H}|_{X'}) \simeq \prod_{\N} \H^{1}(X',\mc{H}|_{X'}) = 0$.

For \labelcref{LemmaKS5-eqn-02}, write \[ \textstyle \mc{F} = \colim_{i \in I} \mc{F}_i \] as a filtered colimit of coherent
1-twisted $\mc{O}_{\ms{X}}$-modules \cite[2.2.1.5]{LIEBLICH-MOTS2007}.
Since $\mc{F}$ is countably generated, we can assume the index set $I$ is countable; after further refinement, we may assume $I = \N$ with the usual ordering, i.e. we have
\[ \mc{F}_1 \to \mc{F}_2 \to \mc{F}_3 \to \dotsb \to \mc{F} \]
with colimit $\mc{F}$. Set $\mc{K}_{i} := \HomS_{\mc{O}_{X}}(\mc{F}_{i},\mc{H})$ and $\mc{K} := \lim_{i \in \N}\mc{K}_{i} \simeq \HomS_{\mc{O}_{X}}(\mc{F},\mc{H})$.

Take an \'etale surjection $W \to X$ where $W = \Spec B$ is affine and such that $\alpha|_{W}$ is trivial;
let $L$ be a 1-twisted invertible $\mc{O}_{\ms{X}_{W}}$-module;
then $\mc{F}' := \mc{F}|_{\ms{X}_{W}} \otimes_{\mc{O}_{\ms{X}_{W}}} L^{-1}$ is a 0-twisted locally projective $\mc{O}_{\ms{X}_{W}}$-module of countably infinite rank, hence $\mc{F}'$ is free by Bass' theorem. By \cite[059Z]{SP} we have that $\mc{F}'$ is a Mittag-Leffler module, hence the system $\{\HomS_{\mc{O}_{W}}(\mc{F}_i|_W, \mc{H}|_W)\}_{i \in \N}$
is Mittag-Leffler by \cite[059E]{SP} (1)$\Rightarrow$(4).
Since $W \to X$ is faithfully flat and $\HomS_{\mc{O}_{W}}(\mc{F}_i|_W, \mc{H}|_W) \simeq \mc{K}_{i}|_{W}$, the system $\{\mc{K}_{i}\}_{i \in \N}$ is Mittag-Leffler; thus we may construct another inverse system $\{\mc{K}_{i}'\}_{i \in \N}$
such that $\mc{K} \simeq \lim_{i \in \N} \mc{K}_{i}'$ and the transition maps $\mc{K}_{i+1}' \to \mc{K}_{i}'$ are surjective.
Then we have $\H^{1}(X,\mc{K}) = \H^{1}(X,\lim \mc{K}_{i}') = 0$ by \cite[0A0J]{SP} (3). 

For the second claim, since every quasi-coherent 1-twisted $\mc{O}_{\ms{X}}$-module $\mc{G}$ is the filtered colimit of coherent 1-twisted $\mc{O}_{\ms{X}}$-submodules, it suffices to show that every coherent 1-twisted $\mc{O}_{\ms{X}}$-module $\mc{G}$ admits a surjection from a direct sum of copies of $\mc{F}$. 
We show that, for a closed point $i : u \to X$, there exists an index set $I$ and a morphism $\mc{F}^{\oplus I} \to \mc{G}$ which is surjective in a neighborhood of $u$. By Nakayama's lemma, this is equivalent to requiring that the fiber $\mc{F}|_{\ms{X}_{u}}^{\oplus I} \to \mc{G}|_{\ms{X}_{u}}$ is surjective. Let $\mc{I}_u \subset \mc{O}_X$ be the ideal sheaf of $i$ and consider the exact sequence
\[ 0 \to \mc{I}_u \mc{G} \to \mc{G} \to \mc{G}/\mc{I}_u\mc{G} \to 0 \] of $\mc{O}_{\ms{X}}$-modules.
By the first claim, we have $\Ext^{1}(\mc{F},\mc{I}_u \mc{G}) = 0$, hence \[\Hom_{\mc{O}_{\ms{X}}}(\mc{F},\mc{G}) \to \Hom_{\mc{O}_{\ms{X}}}(\mc{F},\mc{G}/\mc{I}_{u}\mc{G})\] is surjective. Hence it suffices to
find an $\mc{O}_{\ms{X}}$-linear surjection $\mc{F}^{\oplus I} \to \mc{G}/\mc{I}_{u}\mc{G}$ for some index set $I$, i.e. an $\mc{O}_{\ms{X}_{u}}$-linear surjection $\mc{F}^{\oplus I}/\mc{I}_{u}\mc{F}^{\oplus I} \to \mc{G}/\mc{I}_{u}\mc{G}$. Since $\ms{X}_{u}$ is a $\G_{m}$-gerbe over the spectrum of a field, by Wedderburn's theorem there is a 1-twisted vector bundle $\mc{E}$ on $\ms{X}_{u}$ of minimal positive rank (equal to the index of $\ms{X}_{u}$) and every 1-twisted vector bundle on $\ms{X}_{u}$ is isomorphic to a direct sum of this $\mc{E}$. This gives the desired result. \end{proof}

\begin{lemma}[Eilenberg swindle] \label{LemmaKS4} Let $X$ be a Noetherian affine scheme, let $\ms{X} \to X$ be a $\G_{m}$-gerbe, let $\mc{F},\mc{G}$ be two countably generated locally projective 1-twisted $\mc{O}_{\ms{X}}$-modules of positive rank. Then $\mc{F}^{\oplus \N} \simeq \mc{G}^{\oplus \N}$. \end{lemma} \begin{proof} Set $\mc{F'} := \mc{F}^{\oplus \N}$ and $\mc{G}' := \mc{G}^{\oplus \N}$. By \Cref{LemmaKS5}, there exists an index set $I$ and a surjective $\mc{O}_{\ms{X}}$-linear map $(\mc{F}')^{\oplus I} \to \mc{G}$; since $\mc{G}$ is countably generated, we may assume $I = \N$. Thus we obtain a surjection $\mc{F}' \simeq ((\mc{F}')^{\oplus \N})^{\oplus \N} \to \mc{G}^{\oplus \N} = \mc{G}'$. Since $\mc{G}'$ is a projective object in the category of quasi-coherent 1-twisted $\mc{O}_{\ms{X}}$-modules by \Cref{LemmaKS5}, we obtain a decomposition \[ \mc{G}' \oplus \mc{Q} \simeq \mc{F}' \] which gives an isomorphism \[ \mc{G}' \oplus \mc{F}' \simeq \mc{G}' \oplus (\mc{G}' \oplus \mc{Q}) = (\mc{G}' \oplus \mc{G}') \oplus \mc{Q} \simeq \mc{G}' \oplus \mc{Q} \simeq \mc{F}' \] where we use that $\mc{G}' \oplus \mc{G}' \simeq \mc{G}'$. By symmetry, we have \[ \mc{G}' \simeq \mc{G}' \oplus \mc{F}' \] which implies $\mc{F}' \simeq \mc{G}'$. \end{proof}

\begin{question} \label{0026} In \Cref{LemmaKS4} (which is a twisted analogue of Bass' theorem), can we do without the infinite direct sum, i.e. is it necessarily true that in fact $\mc{F} \simeq \mc{G}$? \end{question}

\begin{remark} \label{LemmaKS4-1} In \Cref{LemmaKS4}, if one of $\mc{F}$ or $\mc{G}$ is assumed to be of finite rank, then we may argue as follows, without the need for \Cref{LemmaKS5}. Say $\mc{G}$ has finite rank; then we have isomorphisms \begin{align*} \mc{F}^{\oplus \N} \simeq \mc{F} \otimes \mc{O}_{X}^{\oplus \N} \stackrel{\ast}{\simeq} \mc{F} \otimes (\mc{G}^{\vee} \otimes \mc{G})^{\oplus \N} \simeq (\mc{F} \otimes \mc{G}^{\vee})^{\oplus \N} \otimes \mc{G} \stackrel{\ast}{\simeq} \mc{O}_{X}^{\oplus \N} \otimes \mc{G} \simeq \mc{G}^{\oplus \N} \end{align*} of $\mc{O}_{\ms{X}}$-modules where those labeled $\ast$ follow from Bass \cite{BASS-BPMAF1963}.  \end{remark}

\begin{theorem} \label{LemmaK} For $X$ as in \Cref{LemmaI}, we have $\LPBr(X) = \H^2_{\et}(X, \G_m)$. \end{theorem}
\begin{proof} Let $\alpha \in \H_{\et}^{2}(X,\G_{m})$ and let $\ms{X} \to X$ be a $\G_{m}$-gerbe corresponding to $\alpha$. By \Cref{LemmaJ} there exists an affine scheme $T$ and a surjective affine-pure morphism $T \to X$ such that $\alpha|_{T}$ is Zariski-locally trivial; moreover, by \Cref{LemmaG}, if $\alpha|_{T} \in \LPBr(T)$ then $\alpha \in \LPBr(X)$. Thus we may reduce to the case when $X$ is affine and $\alpha$ is Zariski-locally trivial. By absolute Noetherian approximation \cite[01ZA, 09YQ]{SP} we may also assume that $X$ is of finite type over $\Z$.

Suppose $X = \Spec A$ and $f_{1},\dotsc,f_{n} \in A$ are elements which generate the unit ideal of $A$ and such that every $\ms{X}_{f_{i}}$ has a countably generated 1-twisted locally projective $\mc{O}_{\ms{X}_{f_{i}}}$-module $\mc{E}_{i}$ of positive rank. We follow an idea of Gabber to produce a countably generated 1-twisted locally projective $\mc{O}_{\ms{X}}$-module of positive rank. Let $a_{1},\dotsc,a_{n} \in A$ be elements such that $a_{1}f_{1} + \dotsb + a_{n}f_{n} = 1$; after replacing $f_{i}$ by $a_{i}f_{i}$, we may assume that $f_{1} + \dotsb + f_{n} = 1$. By \Cref{LemmaKS4} we have $\mc{E}_{1}^{\oplus \N}|_{\ms{X}_{f_{1}} \cap \ms{X}_{f_{2}}} \simeq \mc{E}_{2}^{\oplus \N}|_{\ms{X}_{f_{1}} \cap \ms{X}_{f_{2}}}$ so we may glue to obtain a countably generated 1-twisted locally projective $\mc{O}_{\ms{X}_{f_{1}} \cup \ms{X}_{f_{2}}}$-module. Since $X_{f_{1}+f_{2}} \subseteq X_{f_{1}} \cup X_{f_{2}}$, restriction gives a 1-twisted $\mc{O}_{\ms{X}_{f_{1}+f_{2}}}$-module; then we conclude by induction on $n$. \end{proof}

\begin{remark} In \Cref{LemmaK}, we may also argue using \Cref{LemmaKS4-1}, since we may assume that the $\mc{E}_{i}$ are in fact invertible $\mc{O}_{\ms{X}_{f_{i}}}$-modules: Set $f_{i}' := f_{1} + \dotsb + f_{i}$, and suppose $\mc{E}_{i}'$ is a countably generated 1-twisted locally projective $\mc{O}_{\ms{X}_{f_{i}'}}$-module; then \Cref{LemmaKS4-1} implies that the restrictions of $\mc{E}_{i}'^{\oplus \N}$ and $\mc{E}_{i+1}^{\oplus \N}$ to $\ms{X}_{f_{i}'} \cap \ms{X}_{f_{i+1}}$ are isomorphic. \end{remark}

\begin{corollary} \label{LemmaKp} Let $X$ be a Noetherian scheme admitting a cover $X = U_{1} \cup U_{2}$ such that both $U_{1},U_{2}$ are as in \Cref{LemmaI} and $U_{1} \cap U_{2}$ is affine. Then $\LPBr(X) = \H_{\et}^{2}(X,\G_{m})$. \end{corollary}
\begin{proof} Let $\ms{X} \to X$ be a $\G_{m}$-gerbe. By \Cref{LemmaK}, there exists a countably generated 1-twisted locally projective $\mc{O}_{\ms{X}_{U_{i}}}$-module $\mc{E}_{i}$; by \Cref{LemmaKS4}, there is an isomorphism $\mc{E}_{1}^{\oplus \N}|_{U_{1} \cap U_{2}} \simeq \mc{E}_{2}^{\oplus \N}|_{U_{1} \cap U_{2}}$ on $U_{1} \cap U_{2}$, so we may glue $\mc{E}_{1}^{\oplus \N}$ and $\mc{E}_{2}^{\oplus \N}$.
\end{proof}

\begin{proposition} \label{0021} Let $k$ be a field, let $X_{0}$ be a separated finite type $k$-scheme with the resolution property, let $X$ be a finite order thickening of $X_{0}$. Then $\LPBr(X) = \H^{2}_{\et}(X,\G_{m})$. \end{proposition}
\begin{proof} By \Cref{LemmaG}, it suffices to show that there exists an affine scheme $U$ and an affine-pure morphism $U \to X$. As in \Cref{LemmaI}(e), we may write $X_{0} = W_{0}/\GL_{n}$ for some free action of $\GL_{n}$ on a quasi-affine scheme $W_{0}$; then the Jouanolou trick gives an affine scheme $U_{0}$ and a smooth surjective morphism $U_{0} \to W_{0}$ whose fibers are affine spaces. The composition $U_{0} \to W_{0} \to X_{0}$ is smooth, surjective, affine-pure, and has geometrically irreducible fibers. Since $X$ is affine and $U_{0} \to X_{0}$ is smooth, there are no obstructions to deforming such $U_{0}$ over $X$, i.e. there exists a scheme $U$ admitting a flat morphism $U \to X$ such that $U \times_{X} X_{0} \simeq U_{0}$. Then $U$ is an affine scheme since it is a thickening of an affine scheme. The map $U \to X$ is surjective, smooth (by \cite[06AG]{SP} (17)), and has geometrically irreducible fibers; hence $U \to X$ is affine-pure by \Cref{LemmaH}(a). \end{proof}

\begin{example} \label{0024} Let $k$ be a separably closed field, set $X_{0} := \P_{k}^{2}$, let $\mc{I}$ be a coherent $\mc{O}_{X_{0}}$-module and let $X$ be an infinitesimal thickening of $X_{0}$ by $\mc{I}$. Mathur \cite[Corollary 4]{MATHUR-EOTBMIHC2020} has shown that every $\G_{m}$-gerbe over $X$ corresponding to a torsion class in $\H_{\et}^{2}(X,\G_{m})_{\mr{tors}}$ admits a 1-twisted finite locally free module; our \Cref{0021} (or \Cref{LemmaI}(f) and \Cref{LemmaG}) shows that every $\G_{m}$-gerbe over $X$ admits a 1-twisted locally projective module. In case $k$ has characteristic $0$, we may choose e.g. $\mc{I} = \mc{O}_{X_{0}}(-3)$ as in \cite[Exercise III.5.9]{HARTSHORNE} so that $\H_{\et}^{2}(X,\G_{m})$ does contain non-torsion elements: the exact sequence \[ 1 \to 1+\mc{I} \to \G_{m,X} \to \G_{m,X_{0}} \to 1 \] gives the description $\H_{\et}^{2}(X,\G_{m}) = \H^{2}(X_{0},\mc{I})/\Z = k/\Z$. \end{example}

\begin{remark} \label{0010} Let $X$ be a separated Noetherian scheme. Given any $\alpha \in \H^{2}_{\et}(X,\G_{m})$, we can always find an open subset $U \subset X$ such that $X \setminus U$ has codimension at least 2 in $X$ and such that $\alpha|_{U}$ is contained in $\LPBr(U)$. Indeed, by \cite[09NN]{SP}, we may find two affine open subschemes $U_{1},U_{2}$ of $X$ whose union contains all the codimension 1 points of $X$; then we apply \Cref{LemmaKp}. \end{remark}

\begin{example} \label{LemmaL} The scheme $X$ in \cite[3.11]{EDIDINHASSETTKRESCHVISTOLI-BGAQS2001} also satisfies $\LPBr(X) \ne \H_{\et}^{2}(X,\G_{m})$. We recall the construction. Let $k$ be an algebraically closed field, set $R := k[[x, y, z]]/(xy - z^2)$, let $S_{1} = S_{2} = \Spec R$, let $U$ be the punctured spectrum of $\Spec R$, and let $X$ be the gluing of $S_{1}$ and $S_{2}$ along the identity morphism on $U$. The Mayer-Vietoris exact sequence gives a coboundary map $\partial : \H^{1}_{\et}(U,\G_{m}) \to \H_{\et}^{2}(X,\G_{m})$ which is an isomorphism since $\H_{\et}^{i}(\Spec R,\G_{m}) = 0$ for $i > 0$. Here $\H_{\et}^{1}(U,\G_{m}) = \Pic(U) = \Z/(2)$, and under the isomorphism $\partial$ the unique nonzero class $\alpha \in \H_{\et}^{2}(X,\G_{m})$ corresponds to the invertible $\mc{O}_{U}$-module $\mc{L}$
which (uniquely) extends to the coherent $\mc{O}_{X}$-module corresponding to
the $R$-module $M = \langle x, z \rangle R$.
Let $\ms{X}$ be the $\G_{m}$-gerbe corresponding to $\alpha$. Since $R$ is strictly henselian, the restriction of $\ms{X}$ to $S_{i}$ is trivial; let $\mc{F}_{i}$ be a 1-twisted line bundle on $\ms{X}_{S_{i}}$. Thus if $\alpha$ is contained in $\LPBr(X)$, this would imply there exist 1-twisted locally free $\mc{O}_{\ms{X}|_{S_{i}}}$-modules $\mc{E}_{i}$ and an isomorphism \[ \mc{E}_{1}|_{U} \otimes_{\mc{O}_{\ms{X}_{U}}} \mc{L}|_{\ms{X}_{U}} \simeq \mc{E}_{2}|_{U} \] on $U$. Tensoring by $(\mc{F}_{i})^{-1}$ for either $i=1,2$ gives an $\mc{O}_{U}$-module isomorphism
\[ \mc{L}^{\oplus I} \simeq \mc{O}_{U}^{\oplus J} \]
where $I,J$ are index sets.
By Hartog's theorem, this would give an $R$-module isomorphism
\[ M^{\oplus I} \simeq R^{\oplus J} \]
but no nonempty direct sum of copies
of $M$ is a free $R$-module (since this would imply that $M$ is projective, being a direct summand of a free module). Thus $\H^{2}_{\et}(X,\G_{m}) = \Br'(X) = \Z/(2)$ and $\LPBr(X) = 0$. \end{example}

\begin{example} \label{0011} In \Cref{LemmaL}, we may instead take $R := k[[x,y,z,w]]/(xy-zw)$, in which case $\Pic(U) = \Z$ and so we obtain an example with $\Br(X) = \Br'(X) = \LPBr(X) = 0$ while $\H_{\et}^{2}(X,\G_{m}) = \Z$. \end{example}

\begin{remark} \label{LemmaM} An algebraic space $X$ is said to have the resolution property if every quasi-coherent $\mc{O}_{X}$-module of finite type $\mc{F}$ admits a surjective $\mc{O}_{X}$-linear map from a finite locally free $\mc{O}_{X}$-module. It is an open question to determine whether every quasi-compact separated algebraic space $X$ has the resolution property \cite{TOTARO-TRPFSAS2004}. To find a counterexample, it would be enough (by \Cref{LemmaK} and \Cref{LemmaI}(e)) to find a separated Noetherian scheme $X$ which satisfies $\LPBr(X) \ne \H^2_{\et}(X, \G_m)$. \end{remark}

\section{Very positive vector bundles and pushouts} \label{sec04}

In this section we introduce \emph{very positive vector bundles\/} (\Cref{defn01}), which are vector bundles of infinite rank that are ``infinitely ample'' (to be made precise in \Cref{20210518-10}). These bundles have rather strong uniqueness properties, and we use this to study \Cref{infinite brauer problem} for certain non-projective varieties that arise as pushouts of projective varieties.

\subsection{Pinching and $\LPBr$} In this section we study $\LPBr(X)$ for a class of examples of a proper, non-projective $k$-scheme $X$. We begin with the following example, which is later generalized in \Cref{0019}.

\begin{example} \label{0018} Let $k$ be an algebraically closed field. A standard way to make a nonprojective
proper variety over $k$ is to choose
\begin{enumerate} \item a smooth projective variety $Y$, \item an integer $n > 0$, \item automorphisms $g_i : Y \to Y$ for $i = 1, \dotsc , n$, and \item pairwise distinct points $s_1, \dotsc, s_n, t_1, \dotsc, t_n$ in $\P^1(k)$ \end{enumerate}
and let $X$ be the coequalizer of the diagram \[ \textstyle \coprod_{i=1,\dotsc,n} Y \rightrightarrows \P^{1} \times Y \] where the two maps send $y \mapsto (s_{i},y),(t_{i},g_{i}(y))$ on the $i$th component $Y$. For the existence of the coequalizer $X$, we refer to Ferrand's pinching \cite{FERRAND-CDEP2003}. If there does not exist an ample line bundle $\mc{L}$ on $Y$ such that $g_{i}^{\ast}\mc{L} \simeq \mc{L}$ for all $i$, then $X$ is not projective (for example when $Y$ is an abelian variety and $g_{1}$ is translation by a nontorsion point).

\begin{proposition}
For $X$ as above, $\LPBr(X) = \H^{2}_{\et}(X, \G_{m})$. 
\end{proposition}
\begin{proof}
Let $\nu : \P^1 \times Y \to X$ be the projection; we may identify $\nu$ with the normalization morphism of $X$. Then we have an exact sequence
\[ \textstyle 1 \to \G_{m,X} \to \nu_{\ast}(\G_{m,\P^{1} \times Y}) \to \prod_{i=1,\dotsc,n} \G_{m,Y} \to 1 \]
in the \'etale topology on $X$. The long exact sequence in cohomology gives a map
\begin{align} \label{0018-02} \textstyle \prod_{i = 1, \dotsc, n} \Pic(Y) \to \H^{2}_{\et}(X, \G_m) \end{align} which is surjective since the next terms are $\H_{\et}^{2}(\P^{1} \times Y,\G_{m}) \to \prod_{i=1,\dotsc,n} \H_{\et}^{2}(Y,\G_{m})$ which is the diagonal embedding, in particular injective.

Given line bundles $\mc{L}_{1},\dotsc,\mc{L}_{n}$ on $Y$, the image of $(\mc{L}_{1},\dotsc,\mc{L}_{n})$ under \labelcref{0018-02} is in
$\LPBr(X)$ if and only if there exists a locally projective $\mc{O}_{\P^1 \times Y}$-module $\mc{E}$ such that
\begin{align} \label{0018-01} \mc{E}_{s_i} \simeq g_i^{\ast}(\mc{E}_{t_i}) \otimes_{\mc{O}_{Y}} \mc{L}_i \end{align}
for all $i$.

Let $G \subset \Aut_{k}(Y)$ be the subgroup generated by $g_{1},\dotsc,g_{n}$ and define
\[ \textstyle \mc{V} := \bigoplus_{r \ge 0} \bigoplus_{h_{1},\dotsc,h_{r} \in G} \bigoplus_{1 \le i_{1},\dotsc,i_{r} \le n} \bigoplus_{e_{1},\dotsc,e_{r} \in \Z} (h_1^{\ast} \mc{L}_{i_1}^{e_1} \otimes_{\mc{O}_{Y}} \dotsb \otimes_{\mc{O}_{Y}} h_r^{\ast} \mc{L}_{i_r}^{e_r}) \]
and \[ \mc{E} := \mc{V}|_{\P^{1} \times Y} \] which satisfies \labelcref{0018-01} since $g_{i}^\ast \mc{V} \simeq \mc{V}$ and $\mc{V} \otimes_{\mc{O}_{Y}} \mc{L}_{i} \simeq \mc{V}$ for all $i$. 
\end{proof}
\end{example}

\subsection{Conditions on global generation and vanishing higher cohomology} \label{sec04-02} We generalize the above example by defining certain infinite rank vector bundles on projective $k$-schemes.

\begin{definition} \label{20210518-10} Let $X$ be a scheme. Let \begin{align} \label{20210518-10-eqn-01} \mc{E}_{\bullet} = \{ \mc{E}_{0} \to \mc{E}_{1} \to \mc{E}_{2} \to \dotsb \} \end{align} be a sequence of locally split injections of finite locally free $\mc{O}_{X}$-modules (of positive rank). We define the following conditions. \begin{enumerate} \item[\DEF{$(\mb{G})$}] For any quasi-coherent $\mc{O}_{X}$-module $\mc{F}$ of finite type, there exists some $n' \in \N$ such that $\mc{F} \otimes_{\mc{O}_{X}} \mc{E}_{n}$ is globally generated for all $n \ge n'$. \item[\DEF{$(\mb{G}')$}] For any quasi-coherent $\mc{O}_{X}$-module $\mc{F}$ of finite type, the colimit $\varinjlim_{n \in \N} (\mc{F} \otimes_{\mc{O}_{X}} \mc{E}_{n})$ is globally generated. \item[\DEF{$(\mb{V}_{\ell})$}] For any quasi-coherent $\mc{O}_{X}$-module $\mc{F}$ of finite type, there exists some $n' \in \N$ such that $\H^{q}(X,\mc{F} \otimes_{\mc{O}_{X}} \mc{E}_{n}) = 0$ for all $q \ge \ell+1$ and all $n \ge n'$. \item[\DEF{$(\mb{V}_{\ell}')$}] For any quasi-coherent $\mc{O}_{X}$-module $\mc{F}$ of finite type, we have $\varinjlim_{n \in \N} \H^{q}(X,\mc{F} \otimes_{\mc{O}_{X}} \mc{E}_{n}) = 0$ for all $q \ge \ell+1$. \end{enumerate} \end{definition}

\begin{pg} \label{20210518-16} The primary way of getting a system $\mc{E}_{\bullet}$ satisfying $(\mb{G})$ and $(\mb{V}_{0}')$ is as follows. Let $X$ be a scheme admitting an ample line bundle $\mc{L}$. After replacing $\mc{L}$ by a tensor power, we may assume there exist sections $s_{1},\dotsc,s_{m} \in \Gamma(X,\mc{L})$ such that each $X_{s_{i}}$ is affine and $X = \bigcup_{i=1,\dotsc,m} X_{s_{i}}$. Let $\varphi : \mc{O}_{X} \to \mc{L}^{\oplus m}$ be the map $1 \mapsto (s_{1},\dotsc,s_{m})$; set \[ \mc{E}_{n} := (\mc{L}^{\oplus m})^{\otimes n} \simeq (\mc{L}^{\otimes n})^{\oplus m^{n}} \] and let the transition map $\mc{E}_{n} \to \mc{E}_{n+1}$ be $\varphi \otimes \id_{\mc{E}_{n}}$, which is split when restricted to any $X_{s_{i}}$. We thus obtain a system \[ \mc{E}_{\bullet} = \{\mc{O}_{X} \to \mc{L}^{\oplus m} \to (\mc{L}^{\otimes 2})^{\oplus m^{2}} \to (\mc{L}^{\otimes 3})^{\oplus m^{3}} \to \dotsb \} \]  of locally split injections of finite locally free $\mc{O}_{X}$-modules. The system $\mc{E}_{\bullet}$ satisfies $(\mb{G})$ by \cite[01Q3]{SP} and $(\mb{V}_{0}')$ by \cite[01XR]{SP}. \end{pg}

\begin{lemma} \label{20210518-08} Let $X$ be a quasi-compact semi-separated scheme. The following are equivalent. \begin{enumerate} \item[$(1)$] There exists an ample line bundle $\mc{L}$ on $X$. \item[$(2')$] There exists a system $\mc{E}_{\bullet}$ as in \Cref{20210518-10} satisfying $(\mb{G})$ and $(\mb{V}_{0}')$. \end{enumerate} If $X$ is proper over a Noetherian affine scheme, then $(1)$ and $(2')$ are equivalent to \begin{enumerate} \item[$(2)$] There exists a system $\mc{E}_{\bullet}$ as in \Cref{20210518-10} satisfying $(\mb{G})$ and $(\mb{V}_{0})$. \end{enumerate} \end{lemma} \begin{proof} $(1)\Rightarrow(2')$: See \Cref{20210518-16}. If in addition $X$ is proper over a Noetherian affine scheme, then we have $(\mb{V}_{0})$ by \cite[III, Theorem 5.2]{HARTSHORNE}. \par $(2')\Rightarrow(1)$: We show that for any $x \in X$ there exists a globally generated line bundle $\mc{L}$ and a section $s \in \Gamma(X,\mc{L})$ such that $X_{s}$ is an affine open neighborhood of $x$; then the result follows from \cite[09NC]{SP}. Since $X$ is quasi-compact, we may assume that $x$ is a closed point; let $i : \Spec \kappa(x) \to X$ be the closed immersion. Let $U \subset X$ be an affine open neighborhood of $x$. Set $Z := X \setminus U$ and $Z' := Z \cup \{x\}$, and let $\mc{I},\mc{I}'$ be the ideal sheaves corresponding to the reduced closed subscheme structures on $Z,Z'$ respectively. By \cite[01PG]{SP}, we may write \[ \textstyle \mc{I}' := \varinjlim_{\lambda \in \Lambda} \mc{I}'_{\lambda} \] where each $\mc{I}'_{\lambda}$ is a quasi-coherent ideal sheaf of finite type. We have an exact sequence \[ 0 \to \mc{I}' \to \mc{I} \to \mc{I}/\mc{I}' \to 0 \] where $\mc{I}/\mc{I}'$ is isomorphic to $i_{\ast}\mc{O}_{\Spec \kappa(x)}$. For any $n$ and $\lambda \in \Lambda$, we have an exact sequence \[ 0 \to \mc{I}'_{\lambda} \otimes_{\mc{O}_{X}} \mc{E}_{n} \to \mc{I} \otimes_{\mc{O}_{X}} \mc{E}_{n} \to \mc{I}/\mc{I}'_{\lambda} \otimes_{\mc{O}_{X}} \mc{E}_{n} \to 0 \] on $X$. \par Choose some $n$ and choose a nonzero element $s_{\circ} \in \Gamma(X,\mc{I}/\mc{I}' \otimes_{\mc{O}_{X}} \mc{E}_{n})$; find some $\lambda \in \Lambda$ such that $s_{\circ}$ lifts to some $s_{1} \in \Gamma(X,\mc{I}/\mc{I}'_{\lambda} \otimes_{\mc{O}_{X}} \mc{E}_{n})$; by $(\mb{V}_{0}')$, there exists some $n' \ge n$ so that the image of $s_{1}$ in $\H^{1}(X,\mc{I}'_{\lambda} \otimes_{\mc{O}_{X}} \mc{E}_{n'})$ is $0$. Lift to an element $s_{1}' \in \Gamma(X , \mc{I} \otimes_{\mc{O}_{X}} \mc{E}_{n'})$. Find $n'' \ge n'$ such that $\mc{E}_{n''}$ is globally generated, set $r := \rk \mc{E}_{n''}$, let $s_{1}'' \in \Gamma(X , \mc{I} \otimes_{\mc{O}_{X}} \mc{E}_{n''})$ be the image of $s_{1}'$ and choose sections $s_{2}'',\dotsc,s_{r}'' \in \Gamma(X,\mc{E}_{n''})$ such that the images of $s_{1}'',s_{2}'',\dotsc,s_{r}''$ in $i^{\ast}\mc{E}_{n''}$ constitute a basis for $i^{\ast}\mc{E}_{n''}$ as a $\kappa(x)$-vector space, then consider $s_{1}'' \otimes \dotsb \otimes s_{r}'' \in \Gamma(X,T^{r}\mc{E}_{n''})$ and its image \[ s := s_{1}'' \wedge \dotsb \wedge s_{r}'' \in \Gamma(X,\det \mc{E}_{n''}) \] under the quotient $T^{r}\mc{E}_{n''} \to \det \mc{E}_{n''}$. Then $X_{s}$ is the same as the nonvanishing locus of $s|_{U} \in \Gamma(U,\det \mc{E}_{n''}|_{U})$ in $U$, hence it is affine; by construction, $s$ maps to a generator of $i^{\ast}( \det \mc{E}_{n''})$ as a $\kappa(x)$-vector space, hence $x$ is contained in $X_{s}$. We have that $\det \mc{E}_{n''}$ is globally generated since it is a quotient of the tensor power $T^{r}\mc{E}_{n''}$. \end{proof}

\begin{remark} \label{20210518-14} Assume the notation of \Cref{20210518-10}. \begin{enumerate} \item Condition $(\mb{G})$ implies $(\mb{G}')$, but the converse is not true. Set $X = \P^{1}$ and let $\mc{G}_{\bullet}$ be the system \[ \mc{O}_{X} \to \mc{O}_{X}(1)^{\oplus 2} \to \mc{O}_{X}(2)^{\oplus 4} \to \mc{O}_{X}(3)^{\oplus 8} \to \dotsb \] as in \Cref{20210518-16}, and consider $\mc{E}_{\bullet} := \bigoplus_{n \in \N} \mc{G}_{\bullet}(-1)[-n]$ where ``$[n]$'' denotes shift. Then $\mc{E}_{\bullet}$ satisfies $(\mb{G}')$ but not $(\mb{G})$ since \[ \textstyle \mc{E}_{n} = \bigoplus_{0 \le i \le n} \mc{O}(i-1)^{\oplus 2^{i}} \] is not globally generated for any $n$. \item Condition $(\mb{V}_{\ell})$ implies $(\mb{V}_{\ell}')$, but the converse is not true. For this, set $X = \A^{2} \setminus \{(0,0)\}$ and apply the construction of \Cref{20210518-16}. \end{enumerate} \end{remark}

\begin{lemma} \label{0013} Let $f : Y \to X$ be a closed immersion of Noetherian schemes and let $\mc{E}_{\bullet}$ be a system of vector bundles as in \Cref{20210518-10}. If $\mc{E}_{\bullet}$ satisfies any of $(\mb{G})$, $(\mb{G}')$, $(\mb{V}_{\ell})$, $(\mb{V}_{\ell}')$, then $f^{\ast}\mc{E}_{\bullet}$ satisfies the same property. \end{lemma} \begin{proof} Let $\mc{G}$ be a quasi-coherent $\mc{O}_{Y}$-module of finite type. We have isomorphisms \begin{gather} \label{0013-01} \mc{G} \otimes_{\mc{O}_{Y}} f^{\ast}\mc{E}_{n} \simeq f^{\ast}(f_{\ast}\mc{G} \otimes_{\mc{O}_{X}} \mc{E}_{n}) \\ \label{0013-02} \textstyle\varinjlim_{n \in \N} (\mc{G} \otimes_{\mc{O}_{Y}} f^{\ast}\mc{E}_{n}) \simeq \varinjlim_{n \in \N} f^{\ast}(f_{\ast}\mc{G} \otimes_{\mc{O}_{X}} \mc{E}_{n}) \simeq f^{\ast}(\varinjlim_{n \in \N} (f_{\ast}\mc{G} \otimes_{\mc{O}_{X}} \mc{E}_{n})) \end{gather} and \begin{gather} \label{0013-03} \H^{p}(Y,\mc{G} \otimes_{\mc{O}_{Y}} f^{\ast}\mc{E}_{n}) \simeq \H^{p}(X,f_{\ast}(\mc{G} \otimes_{\mc{O}_{Y}} f^{\ast}\mc{E}_{n})) \stackrel{\dagger}{\simeq} \H^{p}(X,f_{\ast}\mc{G} \otimes_{\mc{O}_{X}} \mc{E}_{n}) \end{gather} where $\dagger$ follows from the projection formula. The claims about $(\mb{G})$ and $(\mb{G}')$ follow from \labelcref{0013-01} and \labelcref{0013-02} because the pullback of a globally generated sheaf is again globally generated, and the claims about $(\mb{V}_{\ell})$ and $(\mb{V}_{\ell}')$ follow from \labelcref{0013-03}. \end{proof}

\subsection{Very positive vector bundles and uniqueness} \label{sec04-03} 

\begin{definition} \label{defn01} Let $k$ be a field, let $X$ be a projective scheme over $k$. Let $\mc{E}$ be a quasi-coherent $\mc{O}_{X}$-module. We say $\mc{E}$ is a
\DEF{very positive vector bundle} if we can write
\[ \mc{E} = \colim (\mc{E}_1 \to \mc{E}_2 \to \mc{E}_3 \to \dotsb) \]
where
\begin{enumerate}
\item[(1)] each $\mc{E}_n$ is a finite rank vector bundle of rank $r_n$,
\item[(2)] $r_1 < r_2 < r_3 < \dotsb$,
\item[(3)] each transition map $\mc{E}_n \to \mc{E}_{n + 1}$ is injective with locally free cokernel, and
\item[(4)] the system $\mc{E}_{\bullet}$ satisfies $(\mb{G})$ and $(\mb{V}_{0})$.
\end{enumerate}
\end{definition}

\begin{lemma} \label{0014} Every projective $k$-scheme $X$ has a very positive vector bundle. \end{lemma}
\begin{proof} This is a consequence of \Cref{20210518-08}. \end{proof}

\begin{remark} In \Cref{defn01}, for every point $x \in X$, the stalk $\mc{E}_{x}$ is a free $\mc{O}_{X,x}$-module by Kaplansky's theorem \cite[3.3]{BASS-BPMAF1963}; however, it is not clear whether $\mc{E}$ is even fppf-locally free. By construction, the very positive vector bundles in \Cref{20210518-08} are Zariski-locally free (of countably infinite rank). \end{remark}

\begin{remark} \label{0022} In general, the pullback of a very positive vector bundle need not be a very positive vector bundle. Indeed, set $X := \P^{1}$ and $Y := \P^{1} \times \P^{1}$, let $f : Y \to X$ be the second projection and let $\mc{E}$ be a very positive vector bundle on $X$. Then $\H^{1}(Y,\mc{O}(-2,0) \otimes_{\mc{O}_{Y}} f^{\ast}\mc{E}_{n}) \ne 0$ for $n \gg 0$ because it contains $\H^{1}(\P^{1},\mc{O}_{\P^{1}}(-2)) \otimes_{k} \H^{0}(X,\mc{E}_{n})$. \end{remark}

\begin{remark} \label{0023} The quotient of a very positive vector bundle need not be a very positive vector bundle. Set $X = \P^{2}$. Let \begin{align} \label{0023-eqn01} \mc{O}_{X} \to \mc{O}_{X}(1)^{\oplus 3} \to \dotsb \to \mc{O}_{X}(n)^{\oplus 3^{n}} \to \dotsb \end{align} be the system obtained by taking $\mc{L} = \mc{O}_{X}(1)$ in \Cref{20210518-16}, let $\{\mc{E}_{n}\}_{n \in \N}$ be the system obtained by twisting \labelcref{0023-eqn01} by $\mc{O}_{X}(-3)$, let $\{\mc{F}_{n}\}_{n \in \N}$ be the constant system defined by $\mc{F}_{n} := \mc{O}_{X}(-3)$ for all $n$, and let $\mc{Q}_{n} := \mc{E}_{n}/\mc{F}_{n}$ be the quotient. Then we have exact sequences \[ \dotsb \to \H^{1}(X,\mc{E}_{n}) \to \H^{1}(X,\mc{Q}_{n}) \to \H^{2}(X,\mc{F}_{n}) \to \H^{2}(X,\mc{E}_{n}) \to \dotsb \] for all $n$. Here $\H^{1}(X,\mc{E}_{n}) = \H^{2}(X,\mc{E}_{n}) = 0$ for $n \gg 0$ but $\H^{2}(X,\mc{F}_{n}) \ne 0$ for all $n$, so $\H^{1}(X,\mc{Q}_{n}) \ne 0$ for all $n \gg 0$; hence $\mc{Q}$ does not satisfy condition (4) in \Cref{defn01}. \end{remark}

\begin{example} \label{20180402-16} As is well-known, there are no nontrivial finite rank vector bundles on the punctured spectrum of a regular local ring of dimension 2. This fact may be used to produce 1-twisted finite rank vector bundles on $\G_{m}$-gerbes over smooth surfaces. Here we explain an example which illustrates a subtlety in the analogous approach for infinite rank twisted sheaves. \par On $\P^{1}$, we have the very positive vector bundle \[ \mc{V} := \colim(\mc{O}_{\P^{1}} \to \mc{O}_{\P^{1}}(1)^{\oplus 2} \to \mc{O}_{\P^{1}}(2)^{\oplus 4} \to \mc{O}_{\P^{1}}(3)^{\oplus 8} \to \dotsb) \] as in \Cref{20210518-16}. Set $U := \A^{2} \setminus \{(0,0)\}$ and let $\pi : U \to \P^{1}$ be the projection map. Then \[ \mc{W} := \pi^{\ast}\mc{V} \simeq \colim(\mc{O}_{U} \to \mc{O}_{U}^{\oplus 2} \to \mc{O}_{U}^{\oplus 4} \to \mc{O}_{U}^{\oplus 8} \to \dotsb) \] is not a free $\mc{O}_{U}$-module; indeed, we have \[ \langle x,y \rangle \cdot \Gamma(U,\mc{W}) = \Gamma(U,\mc{W}) \] since any section in the image of $\Gamma(U,\mc{O}_{U}^{\oplus 2^{i}}) \to \Gamma(U,\mc{W})$ can be written as a linear combination $xs+yt$ for sections $s,t$ in the image of $\Gamma(U,\mc{O}_{U}^{\oplus 2^{i+1}}) \to \Gamma(U,\mc{W})$. Here we use that taking global sections commutes with filtered colimits since $U$ is quasi-compact \cite[0738]{SP}. \par The above also shows that $\mc{V}$ is a locally free $\mc{O}_{\P^{1}}$-module which is not a direct sum of line bundles (otherwise $\mc{W}$ would be as well, but line bundles on $U$ are trivial).  \end{example}

\begin{lemma} \label{0017} Let $k$ be an infinite field, let $X$ be a proper $k$-scheme, let $\mc{E}_{1},\mc{E}_{2},\mc{E}_{3},\mc{F}$ be finite locally free $\mc{O}_{X}$-modules of ranks $r_{1},r_{2},r_{3},s$, respectively. Suppose given an exact sequence \begin{align} \label{0017-01} 0 \to \mc{E}_{1} \stackrel{f_{1}}{\to} \mc{E}_{2} \stackrel{f_{2}}{\to} \mc{E}_{3} \to 0 \end{align} of $\mc{O}_{X}$-modules and let \[ a_{1} : \mc{E}_{1} \to \mc{F} \] be a locally split $\mc{O}_{X}$-linear map. If $\HomS_{\mc{O}_{X}}(\mc{E}_{3},\mc{F})$ is globally generated and $\Ext_{\mc{O}_{X}}^{1}(\mc{E}_{3},\mc{F}) = 0$ and $s > \max\{r_{2},\dim X + r_{3}-1\}$, there exists a locally split $\mc{O}_{X}$-linear map $a_{2} : \mc{E}_{2} \to \mc{F}$ such that $a_{2}f_{1} = a_{1}$. \end{lemma} \begin{proof} Applying $\Hom_{\mc{O}_{X}}(-,\mc{F})$ to \labelcref{0017-01}, the obstruction to the existence of an $a_{2}'$ satisfying $a_{2}'f_{1} = a_{1}$ is an element of $\Ext_{\mc{O}_{X}}^{1}(\mc{E}_{3},\mc{F})$, which is $0$ by assumption; however, such an $a_{2}'$ may not be locally split. Any other $a_{2}''$ satisfying $a_{2}''f_{1} = a_{1}$ is of the form $a_{2}'' = a_{2}' + a_{3}f_{2}$ for a unique $a_{3} : \mc{E}_{3} \to \mc{F}$. If $a_{3}$ is locally split, then $a_{2}''$ is also locally split since the composition $a_{2}''f_{1}$ is locally split (we may locally choose a section of $f_{2}$). Hence we are reduced to the task of producing a morphism $a_{3} : \mc{E}_{3} \to \mc{F}$ which is locally split. \par Set $\mc{H}_{3} := \HomS_{\mc{O}_{X}}(\mc{E}_{3},\mc{F})$ and $H_{3} := \Spec_{X} \Sym_{\mc{O}_{X}}^{\bullet}\mc{H}_{3}^{\vee}$; then $H_{3}$ represents the functor $(\Sch/X)^{\op} \to \Set$ sending $T \mapsto \Hom_{\mc{O}_{T}}(\mc{E}_{3}|_{T},\mc{F}|_{T}) = \Gamma(T,\mc{H}_{3}|_{T})$. Let $\xi_{\mr{univ}} : \mc{E}_{3}|_{H_{3}} \to \mc{F}|_{H_{3}}$ be the universal morphism and let $K \subset H_{3}$ be the closed subscheme defined by the condition that $\xi_{\mr{univ}}|_{T} : \mc{E}_{3}|_{T} \to \mc{F}|_{T}$ is not locally split. The sequence $K \to H_{3} \to X$ is locally on $X$ isomorphic to the pullback of $V_{r_{3}} \to \A_{\Z}^{r_{3}s} \to \Spec \Z$, where $V_{r_{3}}$ is the determinantal variety defined by the maximal minors of a $r_{3} \times s$ matrix with indeterminate coefficients. By Eagon-Northcott \cite[Exercise 10.10]{EISENBUD-CAWAVTAG2004}, the codimension of $V_{r_{3}}$ in $\A_{\Z}^{r_{3}s}$ (and of $K$ in $H_{3}$) is $s-r_{3}+1$. \par Set $N := \dim_{k} \Gamma(X,\mc{H}_{3})$ and let $\mc{O}_{X}^{\oplus N} \to \mc{H}_{3}$ be a surjection; we have a morphism $\A_{k}^{N} \times_{k} X \to H_{3}$ which is locally on $H_{3}$ isomorphic to $\A_{\Z}^{N-r_{3}s} \to \Spec \Z$. Set $I := (\A_{k}^{N} \times_{k} X) \times_{H_{3}} K$; then the codimension of $I$ in $\A_{k}^{N} \times_{k} X$ is also $s-r_{3}+1$, hence $\dim I = N + \dim X - (s-r_{3}+1)$. Since $s > \dim X + r_{3} - 1$, the projection $I \to \A_{k}^{N}$ is not surjective. Since $k$ is infinite, there exists a rational point $p \in \A_{k}^{N}(k)$ which is not in the image of $I \to \A_{k}^{N}$, for which the corresponding fiber $\xi_{\mr{univ}}|_{p} : \mc{E}_{2} \to \mc{F}$ is locally split. \end{proof}

\begin{theorem} \label{0016} Let $k$ be an infinite field, let $X$ be a projective $k$-scheme. Any two very positive vector bundles on $X$ are isomorphic. \end{theorem}
\begin{proof} Suppose
\begin{equation} \label{0016-02}
\begin{aligned}
\mc{E} &= \colim (\mc{E}_1 \to \mc{E}_2 \to \mc{E}_3 \to \dotsb) \\
\mc{F} &= \colim (\mc{F}_1 \to \mc{F}_2 \to \mc{F}_3 \to \dotsb)
\end{aligned}
\end{equation}
are very positive vector bundles on $X$.
We are going to inductively find a sequence of integers $n_{1} < n_{2} < n_{3} < \dotsb$ and maps
\begin{align*}
  a_{2i-1} &: \mc{E}_{n_{2i-1}} \to \mc{F}_{n_{2i}} \\
  a_{2i} &: \mc{F}_{n_{2i}} \to \mc{E}_{n_{2i+1}}
\end{align*}
for $i \ge 1$ such that
\begin{enumerate}
\item[(i)] each $a_{\ell}$ is injective and $\coker(a_{\ell})$ is locally free of positive rank, and
\item[(ii)] the compositions \begin{align*} a_{2i} \circ a_{2i-1} &: \mc{E}_{n_{2i-1}} \to \mc{E}_{n_{2i+1}} \\ a_{2i+1} \circ a_{2i} &: \mc{F}_{n_{2i}} \to \mc{F}_{n_{2i+2}} \end{align*} are equal to the transition maps in \labelcref{0016-02}.
\end{enumerate}

Given a collection of such morphisms $\{a_{\ell}\}_{\ell \in \N}$, condition (ii) implies the existence of morphisms $f : \mc{E} \to \mc{F}$ and $g : \mc{F} \to \mc{E}$ such that $f \circ g = \id_{\mc{F}}$ and $g \circ f = \id_{\mc{E}}$.

For the induction hypothesis, suppose given integers $n_{\ell-1} < n_{\ell}$ and a morphism \[ a_{\ell-1} : \mc{E}_{n_{\ell-1}} \to \mc{F}_{n_{\ell}} \] satisfying (i). (Here, by symmetry we have assumed $\ell$ is even, and we view the base case as a map $a_{0} : 0 \to \mc{E}_{n_{1}}$.) We find an integer $n_{\ell+1}$ and a morphism \[ a_{\ell} : \mc{F}_{n_{\ell}} \to \mc{E}_{n_{\ell+1}} \] satisfying (i) and making the diagram
\begin{equation} \label{0016-01} \begin{tikzpicture}[>=angle 90, baseline=(current bounding box.center)] 
\matrix[matrix of math nodes,row sep=2em, column sep=2em, text height=1.7ex, text depth=0.5ex] { 
|[name=11]| \mc{E}_{n_{\ell-1}} & |[name=12]| \mc{F}_{n_{\ell}} \\ 
|[name=21]| \mc{E}_{n_{\ell+1}} & |[name=22]| \mc{E}_{n_{\ell+1}} \\
}; 
\draw[-,font=\scriptsize,transform canvas={yshift= 0.5pt}](21) edge (22);
\draw[-,font=\scriptsize,transform canvas={yshift=-0.5pt}](21) edge (22);
\draw[->,font=\scriptsize]
(11) edge node[above=0pt] {$a_{\ell-1}$} (12) (11) edge (21) (12) edge[dotted] node[right=0pt] {$a_{\ell}$} (22); \end{tikzpicture} \end{equation} commute.

Set $\mc{Q} := \coker(a_{\ell-1})$. By condition (4) for $\mc{E}$, we may choose $n_{\ell+1} \gg n_{\ell}$ so that $\rk \mc{E}_{n_{\ell+1}} > \max\{\rk \mc{F}_{n_{\ell}},\dim X + \rk \mc{Q} - 1\}$ and $\HomS_{\mc{O}_{X}}(\mc{Q},\mc{E}_{n_{\ell+1}}) \simeq \mc{Q}^{\vee} \otimes_{\mc{O}_{X}} \mc{E}_{n_{\ell+1}}$ is globally generated and $\Ext_{\mc{O}_{X}}^{1}(\mc{Q},\mc{E}_{n_{\ell+1}}) = \H^{1}(\mc{O}_{X},\mc{Q}^{\vee} \otimes \mc{E}_{n_{\ell+1}}) = 0$. Then we conclude using \Cref{0017}.
\end{proof}

\begin{corollary} \label{0020} Let $k$ be an infinite field, let $\mc{E}$ be a very positive vector bundle on a projective $k$-scheme $X$. \begin{enumerate}
\item[(a)] For every finite locally free $\mc{O}_{X}$-module $\mc{V}$ of positive rank, we have $\mc{E} \otimes \mc{V} \simeq \mc{E}$.
\item[(b)] For every invertible $\mc{O}_{X}$-module $\mc{L}$, we have $\mc{E} \otimes \mc{L} \simeq \mc{E}$.
\item[(c)] For every automorphism $g$ of $X$, we have $g^{\ast}\mc{E} \simeq \mc{E}$.
\end{enumerate} \end{corollary} \begin{proof} For (a), it suffices by \Cref{0016} to check that $\mc{E} \otimes \mc{V}$ is a very positive vector bundle. We have \[ \mc{E} \otimes \mc{V} = \colim(\mc{E}_{1} \otimes \mc{V} \to \mc{E}_{2} \otimes \mc{V} \to \mc{E}_{3} \otimes \mc{V} \to \dotsb) \] since tensor products commute with colimits. To show that $\{\mc{E}_{n} \otimes \mc{V}\}_{n \in \N}$ satisfies condition (4) of \Cref{defn01} for a given coherent $\mc{O}_{X}$-module $\mc{F}$, we use that the system $\{\mc{E}_{n}\}_{n \in \N}$ satisfies condition (4) for the coherent $\mc{O}_{X}$-module $\mc{F} \otimes \mc{V}$. \par Claim (b) is a special case of (a), and (c) is clear. \end{proof}

\begin{example} \label{0019} Let $k$ be an infinite field, let $Z, Y$ be smooth projective $k$-schemes, let $a_{1},a_{2} : Z \to Y$ be two closed immersions such that $a_{1}(Z) \cap a_{2}(Z) = \emptyset$. The coequalizer of $a_{1},a_{2}$ exists as a scheme $X$, and the canonical map $\nu : Y \to X$ may be identified with the normalization morphism. 
  
  \begin{proposition} \label{0029}
    For $X$ as above, we have $\LPBr(X)=\H_{\et}^2(X,\G_m)$.
  \end{proposition}
  \begin{proof}
Let $\ms{X} \to X$ be a $\G_{m}$-gerbe with corresponding class $\alpha \in \H_{\et}^{2}(X,\G_{m})$, and let $\beta := \nu^{\ast}\alpha$ and $\gamma := a_{1}^{\ast}\beta = a_{2}^{\ast}\beta$ be the pullbacks to $Y$ and $Z$. Since $Y$ is projective and smooth, we have $\Br(Y) = \Br'(Y) = \H_{\et}^{2}(Y,\G_{m})$ so there exists a 1-twisted finite locally free $\mc{O}_{\ms{X}_{Y}}$-module $\mc{E}$. We have similarly $\Br(Z) = \Br'(Z) = \H_{\et}^{2}(Z,\G_{m})$; let $f : P \to Z$ be a Brauer-Severi scheme corresponding to $\gamma$; we note that $\ms{X}_{P}$ is the trivial $\G_{m}$-gerbe. Consider the following diagram: 
\begin{center} \begin{tikzpicture}[>=angle 90] 
\matrix[matrix of math nodes,row sep=2.5em, column sep=2.5em, text height=1.7ex, text depth=0.5ex] { 
|[name=11]| P & |[name=12]| & |[name=13]|  \\ 
|[name=21]| Z & |[name=22]| Y & |[name=23]| X \\
}; 
\draw[->,font=\scriptsize,transform canvas={yshift= 2.5pt}](21) edge node[above=-1pt] {$a_{1}$} (22);
\draw[->,font=\scriptsize,transform canvas={yshift=-2.5pt}](21) edge node[below=-1pt] {$a_{2}$} (22);
\draw[->,font=\scriptsize] (11) edge node[left=0pt] {$f$} (21) (22) edge node[above=0pt] {$\nu$} (23); \end{tikzpicture} \end{center} 
Let $\mc{V}$ be a very positive vector bundle on $P$; then by \Cref{0020} (a) we have an isomorphism \[ f^{\ast}a_{1}^{\ast}\mc{E} \otimes_{\mc{O}_{\ms{X}_{P}}} \mc{V}|_{\ms{X}_{P}} \simeq f^{\ast}a_{2}^{\ast}\mc{E} \otimes_{\mc{O}_{\ms{X}_{P}}} \mc{V}|_{\ms{X}_{P}} \] of 1-twisted $\mc{O}_{\ms{X}_{P}}$-modules; pushing forward gives an isomorphism \begin{align} \label{0029-01} a_{1}^{\ast}\mc{E} \otimes_{\mc{O}_{\ms{X}_{Z}}} f_{\ast}\mc{V}|_{\ms{X}_{Z}} \simeq a_{2}^{\ast}\mc{E} \otimes_{\mc{O}_{\ms{X}_{Z}}} f_{\ast}\mc{V}|_{\ms{X}_{Z}} \end{align} of 1-twisted $\mc{O}_{\ms{X}_{Z}}$-modules. We have that $f_{\ast}\mc{V}$ is very positive by \Cref{20210518-15}. Let $\mc{W}$ be a very positive vector bundle on $Y$; then $a_{1}^{\ast}\mc{W},a_{2}^{\ast}\mc{W}$ are very positive by \Cref{0013}; hence we have \[ f_{\ast}\mc{V} \simeq a_{1}^{\ast}\mc{W} \simeq a_{2}^{\ast}\mc{W} \] by \Cref{0016}; substituting this into \labelcref{0029-01} gives an isomorphism \[ a_{1}^{\ast}(\mc{E} \otimes_{\mc{O}_{\ms{X}_{Y}}} \mc{W}|_{\ms{X}_{Y}}) \simeq a_{2}^{\ast}(\mc{E} \otimes_{\mc{O}_{\ms{X}_{Y}}} \mc{W}|_{\ms{X}_{Y}}) \] which means $\mc{E} \otimes_{\mc{O}_{\ms{X}_{Y}}} \mc{W}|_{\ms{X}_{Y}}$ descends to $\mc{X}$. \end{proof}
\end{example}

\begin{lemma} \label{20210518-15} Let $k$ be an infinite field, let $f : X \to Y$ be a flat proper morphism between projective $k$-schemes. If $\mc{E}$ is a very positive vector bundle on $X$, then $f_{\ast}\mc{E}$ is a very positive vector bundle on $Y$. \end{lemma} \begin{spg} \label{20210518-15-0002} Let $\mc{O}_{X}(1)$ and $\mc{O}_{Y}(1)$ be globally generated very ample line bundles on $X$ and $Y$, respectively. By \cite[0C4L]{SP}, we have that $\mc{O}_{X}(1)$ is $f$-relatively ample, hence by \cite[0B5T (1), 0EGH]{SP} we may replace $\mc{O}_{X}(1)$ by $\mc{O}_{X}(n)$ for $n \gg 0$ to assume that the map \begin{align} \label{20210518-15-eqn-04} f_{\ast}(\mc{O}_{X}(1)) \otimes_{\mc{O}_{Y}} f_{\ast}(\mc{O}_{X}(m)) \to f_{\ast}(\mc{O}_{X}(m+1)) \end{align} is surjective for all $m \ge 1$. We have \[ f_{\ast}(\mc{O}_{X}(m_{1}) \otimes_{\mc{O}_{X}} f^{\ast}(\mc{O}_{Y}(m_{2}))) \simeq f_{\ast}(\mc{O}_{X}(m_{1})) \otimes_{\mc{O}_{Y}} \mc{O}_{Y}(m_{2}) \] by the projection formula, so after replacing $\mc{O}_{X}(1)$ by $\mc{O}_{X}(1) \otimes_{\mc{O}_{X}} f^{\ast}(\mc{O}_{Y}(m_{2}))$ for some $m_{2} \gg 0$ we may assume that $f_{\ast}(\mc{O}_{X}(1))$ is globally generated (this does not change that \labelcref{20210518-15-eqn-04} is surjective since we tensor both sides by $\mc{O}_{Y}(m_{2})$, and the new line bundle is still ample by \cite[0892]{SP}). By surjectivity of \labelcref{20210518-15-eqn-04}, we have that $f_{\ast}(\mc{O}_{X}(n))$ is globally generated for all $n$. After replacing $\mc{O}_{X}(1)$ by $\mc{O}_{X}(1) \otimes_{\mc{O}_{X}} f^{\ast}(\mc{O}_{Y}(1))$ once more, we may assume that $f_{\ast}(\mc{O}_{X}(n)) \simeq \mc{V}_{n} \otimes_{\mc{O}_{Y}} \mc{O}_{Y}(n)$ for some globally generated coherent $\mc{O}_{Y}$-module $\mc{V}_{n}$. \par By \Cref{20210518-08} and \Cref{0016}, we may assume that $\mc{E} = \varinjlim_{n \in \N} \mc{E}_{n}$ with $\mc{E}_{n} = \mc{O}_{X}(n)^{\oplus r_{n}}$ for some $r_{n}$. \end{spg} \begin{spg}[$f_{\ast}\mc{E}_{n}$ are locally free] Since $\mc{O}_{X}(1)$ is $f$-relatively ample, there exists some $N_{1}$ such that $\mb{R}^{i}f_{\ast}\mc{E}_{n} = 0$ for $n \ge N_{1}$ and $i \ge 1$. By cohomology and base change, we have that $f_{\ast}\mc{E}_{n}$ is finite locally free for $n \ge N_{1}$. \end{spg} \begin{spg}[$f_{\ast}\mc{E}_{n} \to f_{\ast}\mc{E}_{n+1}$ are locally split] Let $\mc{Q}_{n}$ be the cokernel of $\mc{E}_{n} \to \mc{E}_{n+1}$ so that we have an exact sequence \[ 0 \to \mc{E}_{n} \to \mc{E}_{n+1} \to \mc{Q}_{n} \to 0 \] of $\mc{O}_{X}$-modules; the pushforward \[ 0 \to f_{\ast}\mc{E}_{n} \to f_{\ast}\mc{E}_{n+1} \to f_{\ast}\mc{Q}_{n} \to 0 \] is exact since $\mb{R}^{1}f_{\ast}\mc{E}_{n} = 0$. Since $\mc{E}_{n} \to \mc{E}_{n+1}$ is locally split, we have that $\mc{Q}_{n}$ is a finite locally free $\mc{O}_{X}$-module, hence flat over $S$; moreover $\mb{R}^{i}f_{\ast}\mc{Q}_{n} = 0$ for $n \ge N_{1}$ and $i \ge 1$ as well; hence $f_{\ast}\mc{Q}_{n}$ is finite locally free, so $f_{\ast}\mc{E}_{n} \to f_{\ast}\mc{E}_{n+1}$ is locally split for $n \ge N_{1}$. \end{spg} \begin{spg}[$f_{\ast}\mc{E}_{\bullet}$ satisfies $(\mb{V}_{0})$] Let $\mc{G}$ be a coherent $\mc{O}_{Y}$-module. We want to show that $\H^{i}(Y,f_{\ast}\mc{E}_{n} \otimes_{\mc{O}_{X}} \mc{G}) = 0$ for $i \ge 1$ and $n \gg 0$. We have the Leray spectral sequence \[ \mr{E}_{2}^{p,q} = \H^{p}(Y,\mb{R}^{q}f_{\ast}(\mc{E}_{n} \otimes_{\mc{O}_{X}} f^{\ast}\mc{G})) \Rightarrow \H^{p+q}(X,\mc{E}_{n} \otimes_{\mc{O}_{X}} f^{\ast}\mc{G}) \] with differentials $\mr{E}_{2}^{p,q} \to \mr{E}_{2}^{p+2,q-1}$. By Serre vanishing \cite[02O1]{SP} there exists some $N_{2} \ge N_{1}$ such that $\mb{R}^{q}f_{\ast}(\mc{E}_{n} \otimes_{\mc{O}_{X}} f^{\ast}\mc{G}) = 0$ for all $q \ge 1$ and $n \ge N_{2}$. Thus \begin{align} \label{20210518-15-eqn-02}  \H^{p}(Y,f_{\ast}(\mc{E}_{n} \otimes_{\mc{O}_{X}} f^{\ast}\mc{G})) \to \H^{p}(X,\mc{E}_{n} \otimes_{\mc{O}_{X}} f^{\ast}\mc{G}) \end{align} is an isomorphism for all $p \ge 0$ if $n \ge N_{2}$. Since $\mc{E}_{\bullet}$ satisfies $(\mb{V}_{0})$, there exists some $N_{3} \ge N_{2}$ such that $\H^{p}(X,\mc{E}_{n} \otimes_{\mc{O}_{X}} f^{\ast}\mc{G}) = 0$ for all $p \ge 1$ and $n \ge N_{3}$. \par For $n \ge N_{2}$, we have $\mb{R}^{q}f_{\ast}\mc{E}_{n} = 0$ and $\mb{R}^{q}f_{\ast}(\mc{E}_{n} \otimes_{\mc{O}_{X}} f^{\ast}\mc{G}) = 0$ for $q \ge 1$, so the projection formula \cite[08EU]{SP} simplifies to an isomorphism \[ f_{\ast}\mc{E}_{n} \otimes_{\mc{O}_{Y}} \mc{G} \to f_{\ast}(\mc{E}_{n} \otimes_{\mc{O}_{X}} f^{\ast}\mc{G}) \] since $f_{\ast}\mc{E}_{n}$, $\mc{E}_{n}$, and $f$ are flat. Combining this with \labelcref{20210518-15-eqn-02}, we have $\H^{p}(Y,f_{\ast}\mc{E}_{n} \otimes_{\mc{O}_{X}} \mc{G}) = 0$ for $p \ge 1$ and $n \ge N_{3}$. \end{spg} \begin{spg}[$f_{\ast}\mc{E}_{\bullet}$ satisfies $(\mb{G})$] \label{20210518-15-0001} We prove that $f_{\ast}\mc{E}_{n} \otimes_{\mc{O}_{X}} \mc{G}$ is globally generated for $n \gg 0$. Choose a surjection $\mc{O}_{Y}(-N_{4})^{\oplus r} \to \mc{G}$ for some $N_{4}$. Tensoring with $f_{\ast}\mc{E}_{n}$ gives a surjection $f_{\ast}\mc{E}_{n} \otimes_{\mc{O}_{Y}} \mc{O}_{Y}(-N_{4})^{\oplus r} \to f_{\ast}\mc{E}_{n} \otimes_{\mc{O}_{Y}} \mc{G}$; the first term is isomorphic to $(\mc{V}_{n} \otimes_{\mc{O}_{Y}} \mc{O}_{Y}(n-N_{4}))^{\oplus r}$ where $\mc{V}_{n}$ is as in the first paragraph \labelcref{20210518-15-0002}; the latter is globally generated for $n > N_{4}$. \end{spg}

\subsection{Derived category} \label{sec04-04} 

\begin{lemma} \label{20210506-02} Let $k$ be an infinite field, let $X$ be a proper $k$-scheme, let \[ \mc{F} = \colim (\mc{F}_{1} \to \mc{F}_{2} \to \mc{F}_{3} \to \dotsb) \] be a very positive vector bundle on $X$. Suppose given finite rank vector bundles $\mc{E}_{1},\mc{E}_{2}$ and locally split injections $\mc{E}_{1} \to \mc{E}_{2}$ and $\mc{E}_{1} \to \mc{F}_{n_{1}}$ for some $n_{1}$. Then there exists $n_{2} \gg n_{1}$ and a locally split injection $\mc{E}_{2} \to \mc{F}_{n_{2}}$ making the diagram \begin{center} \begin{tikzpicture}[>=angle 90] 
\matrix[matrix of math nodes,row sep=2em, column sep=2em, text height=1.7ex, text depth=0.5ex] { 
|[name=11]| \mc{E}_{1} & |[name=12]| \mc{E}_{2} \\ 
|[name=21]| \mc{F}_{n_{1}} & |[name=22]| \mc{F}_{n_{2}} \\
}; 
\draw[->,font=\scriptsize]
(11) edge (12) (21) edge (22) (11) edge (21) (12) edge[dotted] (22); \end{tikzpicture} \end{center} commute. \end{lemma} \begin{proof} Set $\mc{E}_{3} := \coker(\mc{E}_{1} \to \mc{E}_{2})$. By condition (4) for $\mc{E}$, we may choose $n_{2} \gg n_{1}$ so that $\rk \mc{F}_{n_{2}} > \max\{\rk \mc{E}_{2},\dim S + \rk \mc{E}_{3} - 1\}$ and $\HomS_{\mc{O}_{S}}(\mc{E}_{3},\mc{F}_{n_{2}}) \simeq \mc{E}_{3}^{\vee} \otimes_{\mc{O}_{S}} \mc{F}_{n_{2}}$ is globally generated and $\Ext_{\mc{O}_{S}}^{1}(\mc{E}_{3},\mc{F}_{n_{2}}) = \H^{1}(\mc{O}_{S},\mc{E}_{3}^{\vee} \otimes \mc{F}_{n_{2}}) = 0$. Then we conclude using \Cref{0017}. \end{proof}

\begin{lemma}[``dimension shifting''] \label{20210508-01} Let $k$ be an infinite field, let $X$ be a proper $k$-scheme, let $\ell \ge 1$ and let $\mc{F}_{1}^{[\ell]} \to \mc{F}_{2}^{[\ell]} \to \dotsb$ be a sequence of locally split injections of finite rank vector bundles on $X$ satisfying $(\mb{V}_{\ell})$. Let $\mc{E}$ be a very positive vector bundle on $X$. By \Cref{20210506-02}, there exists an increasing sequence $m_{1} < m_{2} < \dotsb$ and locally split injections $\mc{F}_{i}^{[\ell]} \to \mc{E}_{m_{i}}$ which commute with the transition maps in $\{\mc{F}_{i}^{[\ell]}\}_{i \in \N}$ and $\{\mc{E}_{i}\}_{i \in \N}$. Set $\mc{F}_{i}^{[\ell-1]} := \mc{E}_{m_{i}}/\mc{F}_{i}^{[\ell]}$; then $\{\mc{F}_{i}^{[\ell-1]}\}_{i \in \N}$ satisfies $(\mb{V}_{\ell-1})$. \end{lemma} \begin{proof} Let $\mc{H}$ be a coherent $\mc{O}_{X}$-module. Let $n'$ be such that $i \ge n'$ implies $\H^{q}(X,\mc{F}_{i}^{[\ell]} \otimes_{\mc{O}_{X}} \mc{H}) = 0$ for $q \ge \ell+1$, and let $n''$ be such that $i \ge n''$ implies $\H^{q}(X,\mc{E}_{m_{i}} \otimes_{\mc{O}_{X}} \mc{H}) = 0$ for all $q \ge 1$. We have a (locally split) exact sequence \begin{align} \label{20210508-01-eqn-01} 0 \to \mc{F}_{i}^{[\ell]} \otimes_{\mc{O}_{X}} \mc{H} \to \mc{E}_{m_{i}} \otimes_{\mc{O}_{X}} \mc{H} \to \mc{F}_{i}^{[\ell-1]} \otimes_{\mc{O}_{X}} \mc{H} \to 0 \end{align} for all $i$, hence an exact sequence \[ \H^{q-1}(X,\mc{E}_{m_{i}} \otimes_{\mc{O}_{X}} \mc{H}) \to \H^{q-1}(X,\mc{F}_{i}^{[\ell-1]} \otimes_{\mc{O}_{X}} \mc{H}) \to \H^{q}(X,\mc{F}_{i}^{[\ell]} \otimes_{\mc{O}_{X}} \mc{H}) \] for all $i$. Since $\H^{q}(X,\mc{F}_{i}^{[\ell]} \otimes_{\mc{O}_{X}} \mc{H}) = 0$ for all $i \ge n'$ and $q \ge \ell+1$, if $i \ge \max\{n',n''\}$ then $\H^{q-1}(X,\mc{F}_{i}^{[\ell-1]} \otimes_{\mc{O}_{X}} \mc{H}) = 0$ for $q \ge \ell$. This implies $\mc{F}_{i}^{[\ell-1]}$ satisfies $(\mb{V}_{\ell-1})$. \end{proof}

\begin{proposition} \label{20210508-03} In the setup of \Cref{20210508-01}, assume that $d := \dim X \ge 1$. For any finite rank vector bundle $\mc{F}$ on $X$, there exists an exact sequence \[ 0 \to \mc{F} \to \mc{E}^{[d]} \to \dotsb \to \mc{E}^{[0]} \to 0 \] where each $\mc{E}^{[i]}$ is a very positive vector bundle. \end{proposition} \begin{proof} We view $\mc{F}$ as a constant system $\{\mc{F}_{i}^{[d]}\}_{i \in \N}$; it satisfies $(\mb{V}_{d})$ by \cite[III, Theorem 2.7]{HARTSHORNE}. By \Cref{20210508-01}, we may obtain a sequence of exact sequences \[ 0 \to \mc{F}^{[\ell]} \to \mc{E} \to \mc{F}^{[\ell-1]} \to 0 \] where $\mc{E}$ is a very positive vector bundle and each $\{\mc{F}_{i}^{[\ell]}\}_{i \in \N}$ satisfies $(\mb{V}_{\ell})$. Let $\mc{H}$ be a coherent $\mc{O}_{X}$-module; in view of \labelcref{20210508-01-eqn-01}, whenever $\mc{E}_{m_{i}} \otimes_{\mc{O}_{X}} \mc{H}$ is globally generated, so is $\mc{F}_{i}^{[\ell-1]} \otimes_{\mc{O}_{X}} \mc{H}$. In particular the system $\{\mc{F}_{i}^{[0]}\}_{i \in \N}$ is a very positive vector bundle. \end{proof}

\begin{definition} Let $\mc{D}$ be a triangulated category with arbitrary direct sums. Let $E$ be an
object of $\mc{D}$. \begin{enumerate} 
\item Let \DEF{$\langle E \rangle$} be the strictly full triangulated subcategory of $\mc{D}$ containing $E$ and closed
under taking direct summands.
\item Let \DEF{$\langle E \rangle^{\mr{big}}$} be the strictly full triangulated subcategory of $\mc{D}$ containing $E$ and
closed under taking arbitrary direct sums and under taking direct summands. \end{enumerate} \end{definition}

\begin{lemma} \label{gen-lemma01} Let $X$ be a smooth projective variety over an infinite field $k$, and let $\mc{E}$ be a very positive vector bundle on $X$. Then the essential image of $\mc{D}^b_{Coh}(X) \to \mc{D}^{b}_{QCoh}(X)$ is contained in $\langle \mc{E} \rangle$. \end{lemma}
\begin{proof} Since $X$ is smooth, we can represent any object of $\mc{D}^b_{Coh}(X)$ by a bounded complex of
finite locally free modules. Hence it suffices to show that any finite
locally free $\mc{O}_{X}$-module $F$ is in $\langle \mc{E} \rangle$.
By \Cref{20210508-03} and \Cref{0016} we may find an exact sequence \[ 0 \to F \to \mc{E} \to \dotsb \to \mc{E} \to 0 \] of length $\dim X$. \end{proof}

\begin{proposition} \label{gen-lemma02} Let $X$ be a projective variety over an infinite field $k$, and let $\mc{E}$ be a very positive vector bundle on $X$. Then $\mc{D}_{QCoh}(X) = \langle \mc{E} \rangle^{\mr{big}}$. In particular, $\mc{E}$ is a generator of $\mc{D}_{QCoh}(X)$. \end{proposition}
\begin{proof} It is enough to show that a perfect generator $G$ is contained
in $\langle \mc{E} \rangle^{\mr{big}}$ because we know that $\langle G \rangle^{\mr{big}} = \mc{D}_{QCoh}(X)$. This follows
from the proof of \Cref{gen-lemma01} and \cite[0BQT]{SP}. \end{proof}

\section{Surjective ring map induces surjection on $\GL_{\infty}$} \label{sec05}

In this section we prove that infinite-dimensional invertible matrices lift under any surjective ring map. This surprising fact was discovered while thinking about how to lift twisted vector bundles from a curve to an ambient surface (see \Cref{20210615-17}).

\subsection{Definition and theorem statement}

\begin{definition} \label{20210615-12} Let $I$ be an index set. For any ring $A$, we denote \DEF{$\GL_{I}(A)$} the group of invertible elements of $\Hom_{A}(A^{\oplus I},A^{\oplus I})$. We may identify elements of $\Hom_{A}(A^{\oplus I},A^{\oplus I})$ with $\Mat_{I \times I}^{\mr{cf}}(A)$, matrices whose rows and columns are indexed by $I$ and such that every column has only finitely many nonzero entries; then elements of $\GL_{I}(A)$ correspond to matrices which admit a two-sided inverse. Given a ring homomorphism $\varphi : A \to B$, we obtain a group homomorphism $\GL_{I}(A) \to \GL_{I}(B)$ by applying $\varphi$ to each element in the matrix, and this gives a functor $\mr{Ring} \to \mr{Grp}$. \end{definition}

\begin{theorem} \label{20210615-08} Let $A \to B$ be a surjective ring map. The group homomorphism \[ \GL_{\N}(A) \to \GL_{\N}(B) \] is surjective, i.e. any automorphism of the free $B$-module $B^{\oplus \N}$ lifts to an automorphism of the free $A$-module $A^{\oplus \N}$. \end{theorem}

\subsection{Applications} Before the proof, we discuss some applications.

\begin{remark} \label{20210615-13} Recall that \Cref{20210615-08} is false when $\N$ is replaced by a finite index set $I$; for example if $|I| = 1$, the induced map $\Z^{\times} \to \F_{p}^{\times}$ is not surjective for any prime $p \ge 5$. \end{remark}

\begin{example} Let $A := \Z[x,y]$ and $B := \Z[u^{\pm}]$ and let $A \to B$ be the ring map sending $(x,y) \mapsto (u,u^{-1})$. Then \Cref{20210615-08} implies there exists an invertible matrix $\ml{M} \in \GL_{\N}(A)$ whose image in $\GL_{\N}(B)$ is the diagonal matrix $u \ml{Id}_{\N}$. \end{example}

\begin{remark} One might ask whether there exists a reasonable notion of ``determinant'' for infinite dimensional invertible matrices. By \Cref{20210615-08}, we know at least that there does not exist a natural transformation $\GL_{\N}(-) \to (-)^{\times}$ between functors $\mr{Ring} \to \mr{Grp}$, since every ring $A$ admits a surjection from a polynomial ring $\Z[\{x_{i}\}_{i \in I}]$, but the only units of the latter are $\pm 1$. \end{remark}

\begin{remark} \label{20210615-17} Let $X$ be a separated Noetherian scheme, let $Y \to X$ be a closed subscheme which admits a covering $Y \subset U_{1} \cup U_{2}$ by two affine open subsets of $X$. Then any infinite rank vector bundle on $Y$ extends to $U_{1} \cup U_{2}$. In particular, suppose $X$ is a surface (any quasi-projective $k$-scheme of dimension 2) and $Y$ is a curve in $X$, and let $\mc{E}$ be a countably generated vector bundle on $Y$; then $\mc{E}|_{Y \times_{X} U_{i}}$ are trivial by Bass' theorem; on $Y \times_{X} (U_{1} \cap U_{2})$, the two trivializations differ by a transition map $\varphi \in \GL_{\N}(\Gamma(Y \times_{X} (U_{1} \cap U_{2}) , \mc{O}_{Y}))$ which lifts to some $\varphi' \in \GL_{\N}(\Gamma(U_{1} \cap U_{2} , \mc{O}_{X}))$ by \Cref{20210615-08}. This invertible matrix $\varphi'$ defines a countably generated vector bundle $\mc{E}'$ on $U_{1} \cup U_{2}$ whose restriction to $Y$ is $\mc{E}$. Here we may choose $U_{1},U_{2}$ suitably so that the complement $X \setminus (U_{1} \cup U_{2})$ consists of a finite collection of closed points. Furthermore the vector bundle $\mc{E}'$ may not extend to the entire surface $X$ (see \Cref{20180402-16}). \end{remark}

\begin{question} Is the twisted analogue of \Cref{20210615-08} true? Namely, let $Y \to X$ be a closed immersion of affine schemes, let $\mc{G} \to X$ be a $\G_{m}$-gerbe, and let $\mc{E}$ be an infinite rank 1-twisted vector bundle on $\mc{G}$. Is the map \[ \Aut_{\mc{O}_{\mc{G}}}(\mc{E}) \to \Aut_{\mc{O}_{\mc{G}_{Y}}}(\mc{E}|_{\mc{G}_{Y}}) \] surjective? \end{question}

\subsection{Proof} We begin the proof of \Cref{20210615-08}.

\begin{definition} \label{20210615-05} Let $I$ be an index set. An \DEF{elementary matrix} indexed by $I$ over a ring $A$ is a matrix in $\Mat_{I \times I}(A)$ of the form $\id_{I} + \ml{M}$ where $\ml{M}$ is a matrix in $\Mat_{(I \setminus J) \times J}(A) \simeq \Hom_{A}(A^{\oplus J} , A^{\oplus I \setminus J})$ for some subset $J \subseteq I$. \end{definition}

\begin{remark} \label{20210615-06} Any elementary matrix $\id_{I} + \ml{M}$ is invertible; its inverse is $\id_{I} - \ml{M}$ which is itself an elementary matrix. \end{remark}

\begin{remark} \label{20210615-14} Let $I$ be an index set; for any $i \in \N$, let $\ml{e}_{i} \in \Z^{\oplus \N}$ denote the $i$th basis vector; let $\sigma : I \to I$ be a bijection and let $\varphi_{\sigma} \in \GL_{I}(\Z)$ be the \DEF{permutation matrix} corresponding to $\sigma$, which sends $\ml{e}_{i} \mapsto \ml{e}_{\sigma(i)}$. Since permutation matrices are defined over $\Z$, they lift via any ring map. \end{remark}

\begin{lemma} \label{20210615-07} Let $A$ be a ring, let $\BM{a_{1} & \dotsb & a_{n} } \in A^{\oplus n}$ be a unimodular vector. There exist elementary matrices $\ml{A}_{1},\dotsc,\ml{A}_{m} \in \GL_{n+1}(A)$ such that \[ \ml{A}_{1} \dotsb \ml{A}_{m} \BM{ a_{1} & \dotsb & a_{n} & 0}^{T} = \BM{ 1 & 0 & \dotsb & 0}^{T} \] in $\GL_{n+1}(A)$. \end{lemma} \begin{proof} If $c_{1}a_{1} + \dotsb + c_{n}a_{n} = 1$, perform the column operations $\mr{C}_{n+1} \texttt{ += } c_{i}\mr{C}_{i}$ for $i=1,\dotsc,n$ (after which $\mr{C}_{n+1} = 1$) and $\mr{C}_{i} \texttt{ -= } a_{i}\mr{C}_{n+1}$ for $i=1,\dotsc,n$. \end{proof}

\begin{pg}[Proof of \Cref{20210615-08}] We say that a matrix $\ml{M} \in \GL_{I}(B)$ is \DEF{liftable} if it is in the image of $\GL_{I}(A) \to \GL_{I}(B)$. Any elementary matrix $\id_{I} + \ml{M} \in \GL_{I}(B)$ lifts to an elementary matrix in $\GL_{I}(A)$ (here we must be careful to lift any $0$s as $0$, to ensure that every column contains only finitely many nonzero entries). We note that if $\ml{M}' \in \GL_{I}(A)$ is a lift of $\ml{M} \in \GL_{I}(B)$, then the inverse $\ml{M}'^{-1}$ is a lift of $\ml{M}^{-1}$; furthermore, if $\ml{M}_{1},\ml{M}_{2} \in \GL_{I}(B)$ are invertible matrices such that $\ml{M}_{1}$ is liftable, then $\ml{M}_{2}$ is liftable if and only if their product $\ml{M}_{1} \ml{M}_{2}$ is liftable.

\begin{spg} \label{20210615-10} Let $\ml{P} \in \GL_{\N}(B)$ be an automorphism of $B^{\oplus \N}$. Write $\ml{P}(\ml{e}_1) = \sum_{i = 1, \dotsc, n} a_i \ml{e}_i$ for $a_{i} \in B$. Then we see that $(a_1, \dotsc, a_n)$ is a unimodular vector. By \Cref{20210615-07} we can find a liftable invertible matrix $\ml{T} \in \GL_{n+1}(B)$ such that the operator $\mr{diag}(\ml{T}, 1, 1, \dotsc) \in \GL_{\N}(B)$ sends $\ml{e}_1$ to the same vector as $\ml{P}$. \end{spg}

\begin{spg} \label{20210615-11} Repeating the argument of \Cref{20210615-10}, we can find a sequence of positive integers $r_1, r_2, r_3, \dotsc$ and liftable invertible $r_i \times r_i$ matrices $\ml{T}_i \in \GL_{r_{i}}(B)$ such that $\mr{diag}(\ml{T}_1, \ml{T}_2, \ml{T}_3, \dotsc) \in \GL_{\N}(B)$ sends $\ml{e}_1, \ml{e}_{r_1 + 1}, \ml{e}_{r_1 + r_2 + 1}, \dotsc$ to the same vectors as $\ml{P}$ does. Moreover $\mr{diag}(\ml{T}_1, \ml{T}_2, \ml{T}_3, \dotsc)$ is liftable: if $\ml{T}_{i}' \in \GL_{r_{i}}(A)$ is a lift of $\ml{T}_{i}$, then $\mr{diag}(\ml{T}_1', \ml{T}_2', \ml{T}_3', \dotsc) \in \GL_{\N}(A)$ is a lift of $\mr{diag}(\ml{T}_1, \ml{T}_2, \ml{T}_3, \dotsc)$. \end{spg}

\begin{spg} \label{20210615-16} By \Cref{20210615-11}, after replacing $\ml{P}$ by $\mr{diag}(\ml{T}_1, \ml{T}_2, \ml{T}_3, \dotsc)^{-1} \cdot \ml{P}$, we may assume $\ml{P}(\ml{e}_i) = \ml{e}_i$ for infinitely many $i \in \N$. After possibly rearranging the $\ml{e}_{i}$, we may assume there exists some countable set $I$ such that $\ml{P}$ may be written in block matrix form as $\ml{P} = \BM{\ml{Id}_{\N} & \ml{T} \\ 0 & \ml{U}}$ for some $\ml{T} \in \Mat_{\N \times I}(B)$ and $\ml{U} \in \GL_{I}(B)$. If we can show that $\mr{diag}(\ml{Id}_{\N}, \ml{U}) = \BM{\ml{Id}_{\N} & 0 \\ 0 & \ml{U}}$ is liftable, then we are done since $\BM{\ml{Id}_{\N} & \ml{T} \\ 0 & \ml{U}} \BM{\ml{Id}_{\N} & 0 \\ 0 & \ml{U}^{-1}} = \BM{\ml{Id}_{\N} & \ml{T}\ml{U}^{-1} \\ 0 & \ml{Id}_{I}}$ is liftable (because it is an elementary matrix). \end{spg}

\begin{spg}[Whitehead's lemma] \label{20210615-15} Given invertible matrices $\ml{A},\ml{B} \in \GL_{I}(A)$, we have that $\BM{\ml{A} & 0 \\ 0 & \ml{B}}$ is liftable if and only if $\BM{\ml{A}\ml{B} & 0 \\ 0 & \ml{Id}_{I}}$ is liftable. Namely, perform the following (block) row/column operations by multiplying by appropriate elementary matrices and permutation matrices on the left/right, respectively \cite[I, 5.1]{LAM-SPOPM2006}: \[ \BM{\ml{A} & 0 \\ 0 & \ml{B}} \stackrel{\mr{C}_{1} \texttt{+=} \mr{C}_{2}\ml{B}^{-1}}{\to} \BM{\ml{A} & 0 \\ \ml{Id}_{I} & \ml{B}} \stackrel{\mr{C}_{2} \texttt{-=} \mr{C}_{1}\ml{B}}{\to} \BM{\ml{A} & -\ml{A}\ml{B} \\ \ml{Id}_{I} & 0} \stackrel{\substack{\mr{C}_{1} \Leftrightarrow \mr{C}_{2} \\ \mr{C}_{1} \texttt{*=} -1}}{\to} \BM{\ml{A}\ml{B} & \ml{A} \\ 0 & \ml{Id}_{I}} \stackrel{\mr{R}_{1} \texttt{-=} \ml{A} \mr{R}_{2}}{\to} \BM{\ml{A}\ml{B} & 0 \\ 0 & \ml{Id}_{I}} \] \end{spg}

\begin{spg}[Eilenberg swindle] Let $I$ and $\ml{U}$ be as in \Cref{20210615-16}. If $I$ is empty, there is nothing to show. Suppose $I$ is nonempty; by \Cref{20210615-15} we see that \[ \ml{D} := \mr{diag}(\dotsc, \ml{U}, \ml{U}^{-1}, \ml{U}, \ml{U}^{-1}) \] is liftable; hence we see that \[ \mr{diag}(\ml{Id}_{\N}, \ml{U}) \] is liftable if and only if \[ \mr{diag}(\ml{D},\ml{U}) = \mr{diag}(\dotsc, \ml{U}, \ml{U}^{-1}, \ml{U}, \ml{U}^{-1}, \ml{U}) \] is liftable, which it is by the same argument with $\ml{U}$ replaced by $\ml{U}^{-1}$. \end{spg}
\end{pg}

\appendix
\section{Infinite matrix rings} \label{sec06}

\subsection{Skolem-Noether}
In this section we prove the following theorem, which is closely related to a theorem of Courtemanche--Dugas \cite[2.2]{COURTEMANCHEDUGAS-AOTEAOAFM2016}. We remove their hypotheses on idempotents and indecomposable projective modules, but require that the automorphism is of the entire sheaf of algebras (as opposed to just the global sections). For the finite rank case, see for example \cite[IV, 1.4]{MILNE-EC}.

\begin{theorem} \label{20180329-aaga} Let $\mc{C}$ be a locally ringed site with a final object $S$, let $I$ be an index set, let $V := \bigoplus_{i \in I} \mc{O}_{\mc{C}}\ml{e}_{i}$ be a free $\mc{O}_{\mc{C}}$-module with basis $\{\ml{e}_{i}\}_{i \in I}$, set $\mc{A} := \EndS_{\mc{O}_{\mc{C}}}(V)$ and let $\varphi : \mc{A} \to \mc{A}$ be an $\mc{O}_{\mc{C}}$-algebra automorphism. Then there exists a covering $\{S_{\xi} \to S\}_{\xi \in \Xi}$ and units $\ml{U}_{m} \in \Gamma(S_{\xi},\mc{A}^{\times})$ such that $\varphi|_{S_{\xi}} : \mc{A}|_{S_{\xi}} \to \mc{A}|_{S_{\xi}}$ is the conjugation-by-$\ml{U}_{m}$ morphism. \end{theorem} \begin{proof} Let $\ml{E}_{i} \in \Gamma(S,\mc{A})$ be the projection onto the $i$th summand, and let $W_{i} \subset V$ be the image of $\ml{E}_{i}$. The formation of $W_{i}$ is compatible with localization since the projections $\ml{E}_{i}$ are. By \Cref{20210514-02}, the $W_{i}$ are pairwise-isomorphic invertible $\mc{O}_{\mc{C}}$-modules. Since $\mc{C}$ is locally ringed, there exists a covering $\{S_{\xi} \to S\}_{\xi \in \Xi}$ such that each $W_{i}|_{S_{\xi}}$ is a trivial $\mc{O}_{S_{\xi}}$-module; then we conclude by \Cref{20210514-09}. \end{proof}

\begin{lemma} \label{20210514-02} Let $R$ be a ring, let $I$ be an index set, let $V := \bigoplus_{i \in I} R\ml{e}_{i}$ be a free $R$-module with basis $\{\ml{e}_{i}\}_{i \in I}$, set $\mc{A} := \End_{R}(V)$ and let $\varphi : \mc{A} \to \mc{A}$ be an $R$-algebra automorphism. Let $\ml{E}_{i} \in \mc{A}$ be the projection onto the $i$th summand, let $q_{i} := \varphi(\ml{E}_{i})$ be the image of $\ml{E}_{i}$, and let $W_{i} := q_{i}(V)$ be the image of $q_{i}$. Then \begin{enumerate} \item[(i)] the $W_{i}$ are pairwise isomorphic, \item[(ii)] each $W_{i}$ is an invertible $R$-module, and \item[(iii)] the natural map \begin{align} \label{20180329-aaga-eqn-01} \textstyle \bigoplus_{i \in I} W_{i} \to V \end{align} is an isomorphism. \end{enumerate} \end{lemma} 

\begin{spg} For $i_{1},i_{2} \in I$, let $\ml{E}_{i_{1},i_{2}}$ be the $(i_{1},i_{2})$th matrix unit (the only nonzero entry is a 1 in the $(i_{1},i_{2})$th entry) so that $\ml{E}_{i} = \ml{E}_{i,i}$. For any unit $\ml{U} \in \mc{A}^{\times} \simeq \GL_{I}(R)$ we have \[ \varphi(\ml{U}^{-1} \cdot \ml{E}_{i} \cdot \ml{U}) = (\varphi(\ml{U}))^{-1} \cdot q_{i} \cdot \varphi(\ml{U}) \] in $\mc{A}$. Since the $\ml{E}_{i}$ are pairwise conjugate (let $\ml{U}_{i_{1},i_{2}} \in \mc{A}^{\times}$ be the $R$-automorphism switching $\ml{e}_{i_{1}}$ and $\ml{e}_{i_{2}}$; then $\ml{E}_{i_{2}} = \ml{U}_{i_{1},i_{2}}^{-1} \cdot \ml{E}_{i_{1}} \cdot \ml{U}_{i_{1},i_{2}}$), the $q_{i}$ are also pairwise conjugate; namely $(\varphi(\ml{U}_{i_{1},i_{2}}))^{-1}$ defines an isomorphism $W_{i_{1}} \simeq W_{i_{2}}$ by \Cref{20180329-aagz}. This proves (i). \end{spg} 

\begin{spg} \label{20210514-08} Let \[ \ml{N}_{i} \in \Mat_{I \times I}^{\mr{cf}}(R) \] be the (column-finite) infinite matrix corresponding to $q_{i}$. Then $W_{i}$ is equal to the column space $\operatorname{col}(\ml{N}_{i})$. Since $W_{i}$ is the image of an idempotent endomorphism of $V$, we have that $W_{i}$ is a direct summand of $V$, hence $W_{i}$ is a projective $R$-module. \end{spg} 

\begin{spg} \label{20210514-03} Given a matrix $\ml{N} \in \Mat_{I \times I}^{\mr{cf}}(R)$ and any subsets $J_{1},J_{2} \subset I$, let $\ml{N}|_{J_{1},J_{2}}$ denote the $|J_{1}| \times |J_{2}|$ submatrix of $\ml{N}$ associated to $J_{1},J_{2}$. \par Let $\mf{m}$ be a maximal ideal of $R$. By Kaplansky \cite[3.3]{BASS-BPMAF1963}, we have that $W_{i} \otimes_{R} R_{\mf{m}}$ is a free $R_{\mf{m}}$-module. We show that $\rk_{R_{\mf{m}}}(W_{i} \otimes_{R} R_{\mf{m}}) \le 1$. It is enough to show $\dim_{R/\mf{m}} (W_{i} \otimes_{R} R/\mf{m}) \le 1$. Suppose $\dim_{R/\mf{m}} (W_{i} \otimes_{R} R/\mf{m}) \ge 2$; then there exist finite subsets $J_{1},J_{2} \subset I$ of size $|J_{1}| = |J_{2}| = 2$ such that $\det(\ml{N}_{i}|_{J_{1},J_{2}}) \not\in \mf{m}$. After permuting rows and columns, we may assume that $J := J_{1} = J_{2}$. \par Write $\mf{m}$ as a directed colimit $\mf{m} = \varinjlim_{\lambda \in \Lambda} \mf{a}_{\lambda}$ of finitely generated ideals $\mf{a}_{\lambda}$ of $R$; then $R/\mf{m} \simeq \varinjlim_{\lambda \in \Lambda} R/\mf{a}_{\lambda}$ is a directed colimit of finitely presented $R$-algebras and we have an isomorphism \[ \textstyle W_{i} \otimes_{R} R/\mf{m} \simeq \varinjlim_{\lambda \in \Lambda} (W_{i} \otimes_{R} R/\mf{a}_{\lambda}) \] as $R$-modules. There exists $\lambda \in \Lambda$ such that $\det(\ml{N}_{i}|_{J,J})$ is a unit of $R/\mf{a}_{\lambda}$. After multiplying by a matrix whose $J \times J$ submatrix is $(\ml{N}_{i}|_{J,J})^{-1}$, we may assume that $\ml{N}_{i}|_{J,J} = \id_{2}$. By doing row operations and column operations, we may assume that in fact the $i$th row of $\ml{N}_{i}$ for $i \in J$ is the $i$th standard basis vector of $V$ (and similarly for columns). Thus we have a splitting \[ W_{i} \otimes_{R} R/\mf{a}_{\lambda} \simeq R/\mf{a}_{\lambda} \oplus R/\mf{a}_{\lambda} \oplus W_{i}' \] for some submodule $W_{i}' \subset W_{i} \otimes_{R} R/\mf{a}_{\lambda}$. \par Let $q_{i} = q_{i}' + q_{i}''$ be the decomposition of $q_{i}$ (over $R/\mf{a}_{\lambda}$) corresponding to projections onto $R/\mf{a}_{\lambda}$ and $R/\mf{a}_{\lambda} \oplus W_{i}'$, respectively. Since $\mf{a}_{\lambda}$ is finitely generated, we have by \Cref{20180329-aafx} an induced $R/\mf{a}_{\lambda}$-algebra automorphism of $\End_{R/\mf{a}_{\lambda}}(V \otimes_{R} R/\mf{a}_{\lambda})$. Then \Cref{20180329-aagd} implies that this decomposition comes from a pair of orthogonal idempotents on $R/\mf{a}_{\lambda}$, contradiction. \end{spg} 

\begin{spg} \label{20210514-05} We show that $W_{i}$ is finitely generated. We know from \labelcref{20210514-03} that $W_{i} \otimes_{R} R_{\mf{m}}$ is free of rank either $0$ or $1$. In particular $W_{i}$ is finitely generated on a faithfully flat cover of $R$, so $W_{i}$ is finitely generated by \cite[03C4]{SP}. \end{spg} 

\begin{spg} \label{20210514-06} We show that $W_{i} \otimes_{R} R_{\mf{m}}$ is nonzero. It is enough to show that $W_{i} \otimes_{R} R/\mf{m}$ is nonzero. Suppose that there is a maximal ideal $\mf{m}$ containing all the coefficients of $\ml{N}_{i}$. We have that $W_{i}$ is finitely generated by \Cref{20210514-05}, hence the ideal $\mf{a}$ generated by all the coefficients of $\ml{N}_{i}$ is finitely generated. Thus there exist matrices $\ml{K}_{1},\dotsc,\ml{K}_{\ell} \in \Mat_{I \times I}^{\mr{cf}}(R)$ and $x_{1},\dotsc,x_{\ell} \in \mf{m}$ such that $\ml{N}_{i} = x_{1}\ml{K}_{1} + \dotsb + x_{\ell}\ml{K}_{\ell}$. Since $\varphi$ is $R$-linear, the matrix corresponding to $\ml{E}_{i} = \varphi^{-1}(q_{i})$ is $x_{1}\varphi^{-1}(\ml{K}_{1}) + \dotsb + x_{\ell}\varphi^{-1}(\ml{K}_{\ell})$, in particular it has coefficients contained in $\mf{m}$, contradiction. \end{spg}

\begin{spg} \label{20210514-07} Combining \labelcref{20210514-08}, \labelcref{20210514-03}, \labelcref{20210514-05}, \labelcref{20210514-06} gives (ii). \end{spg}

\begin{spg} \label{20210514-04} For (iii), we recall the argument of \cite[2.2]{COURTEMANCHEDUGAS-AOTEAOAFM2016}. We first show that, for a fixed $i_{0} \in I$, we have $q_{j}(\ml{e}_{i_{0}}) = 0$ for all but finitely many $j \in I$ (i.e. there are only finitely many nonzero columns in $\ml{N}_{i_{0}}$). If $q_{j}(\ml{e}_{i_{0}}) \ne 0$, then $q_{j} \cdot \ml{E}_{i_{0}} \ne 0$; thus $\ml{E}_{j} \cdot \varphi^{-1}(\ml{E}_{i_{0}}) \ne 0$; hence the only possibilities are for $j$ in the support of $\varphi^{-1}(\ml{E}_{i_{0}})$. By \labelcref{20210514-05} applied to $\varphi^{-1}$, we have that $\im (\varphi^{-1}(\ml{E}_{i_{0}}))$ is finitely generated, hence its support is finite. \end{spg} 

\begin{spg} The injectivity of \labelcref{20180329-aaga-eqn-01} follows from the fact that the $q_{i}$ are idempotent and $q_{i_{1}} \cdot q_{i_{2}} = 0$ if $i_{1} \ne i_{2}$. For surjectivity, it suffices to show that for every $i_{0} \in I$ the basis element $\ml{e}_{i_{0}}$ is in the image of \labelcref{20180329-aaga-eqn-01}. By \labelcref{20210514-04}, the set $S_{i_{0}} := \{i \in I \;:\; q_{i}(\ml{e}_{i_{0}}) \ne 0\}$ is finite. Set $v_{i_{0}} := \ml{e}_{i_{0}} - \sum_{i \in S_{i_{0}}} q_{i}(\ml{e}_{i_{0}})$. If $i \in I \setminus S_{i_{0}}$, then $q_{i}(v_{i_{0}}) = 0$; if $i \in S_{i_{0}}$, then $q_{i}(v_{i_{0}}) = q_{i}(\ml{e}_{i_{0}}) - q_{i}(\ml{e}_{i_{0}}) = 0$ since $q_{i_{1}} \cdot q_{i_{2}} = 0$ if $i_{1} \ne i_{2}$; thus $v_{i_{0}} \in \bigcap_{i \in I} (\ker q_{i})$. \par It remains to show that $\bigcap_{i \in I} (\ker q_{i}) = 0$. Suppose $v \in V$ is an element such that $q_{i}(v) = 0$ for all $i \in I$; define $\pi : V \to V$ sending $\ml{e}_{i} \mapsto v$ for all $i \in I$; then $q_{i} \cdot \pi = 0$ for all $i \in I$; then $\ml{E}_{i} \cdot \varphi^{-1}(\pi) = 0$ for all $i \in I$, i.e. $\varphi^{-1}(\pi) = 0$, i.e. $\pi = 0$. \end{spg}

\begin{pg} \label{20180329-aagd} Let $R$ be a ring, let $V = \bigoplus_{i \in I} R\ml{e}_{i}$ be a free $R$-module, set $\mc{A} := \End_{R}(V)$, and let $\varphi : \mc{A} \to \mc{A}$ be an $R$-algebra isomorphism. Let $p_{i} \in \mc{A}$ be the projection onto the $i$th summand and set $q_{i} := \varphi(p_{i})$. Suppose $\im q_{i}$ has a decomposition; then $q_{i} = q_{i}' + q_{i}''$ for some nonzero orthogonal idempotents $q_{i}',q_{i}''$; then applying $\varphi^{-1}$ gives a decomposition of $p_{i} = p_{i}' + p_{i}''$ into nonzero orthogonal idempotents; then there exist nonzero idempotents $u',u'' \in R$ such that $u' + u'' = 1$ and $u' u'' = 0$ and $p_{i}' = u_{i}' p_{i}$ and $p_{i}'' = u_{i}'' p_{i}$; since $\varphi$ is an $R$-linear, we have $q_{i}' = u_{i}' q_{i}$ and $q_{i}'' = u_{i}'' q_{i}$ also. \end{pg}

\begin{lemma} \label{20210514-09} In the notation of \Cref{20210514-02}, if each $W_{i}$ is a free $R$-module, then there exists some $\ml{U} \in \mc{A}^{\times}$ such that $\varphi$ is equal to the conjugation-by-$\ml{U}$ map. \end{lemma} \begin{proof}[Proof (from \cite{COURTEMANCHEDUGAS-AOTEAOAFM2016})] Choose generators $\ml{w}_{i} \in W_{i}$ and define $\ml{U}' : V \to V$ sending $\ml{e}_{i} \mapsto \ml{w}_{i}$ for all $i \in I$. Since \labelcref{20180329-aaga-eqn-01} is an isomorphism, we have that $\ml{U}'$ is an isomorphism. Given $u = \sum_{i \in I} u_{i}\ml{e}_{i}$, we have $(q_{i_{0}}\ml{U}')(u) = q_{i_{0}}(\sum_{i \in I} u_{i} \ml{w}_{i}) = u_{i_{0}} \ml{w}_{i_{0}}$ and $(\ml{U}' \ml{E}_{i_{0}})(u) = \ml{U}'(u_{i_{0}} \ml{e}_{i_{0}}) = u_{i_{0}} \ml{w}_{i_{0}}$; hence $q_{i} = \ml{U}' \ml{E}_{i} \ml{U}'^{-1}$ for all $i \in I$. After replacing $\varphi$ by $\varphi \cdot (\ml{U}'^{-1}(-)\ml{U}')$, we may assume that \begin{align} \label{20180329-aaga-eqn-02} \varphi(\ml{E}_{i}) = \ml{E}_{i} \end{align} for all $i \in I$. \par Fix $i_{1},i_{2}$ and set $\varphi(\ml{E}_{i_{1},i_{2}})(\ml{e}_{j_{1}}) = \sum_{j_{2} \in I} s^{i_{1},i_{2}}_{j_{1},j_{2}}\ml{e}_{j_{2}}$ for some $s^{i_{1},i_{2}}_{j_{1},j_{2}} \in R$. Given $u = \sum_{i \in I} u_{i}\ml{e}_{i}$, we have \[ \textstyle (\ml{E}_{i_{1}} \cdot \varphi(\ml{E}_{i_{1},i_{2}}) \cdot \ml{E}_{i_{2}})(u) = (\ml{E}_{i_{1}} \cdot \varphi(\ml{E}_{i_{1},i_{2}}))(u_{i_{2}}\ml{e}_{i_{2}}) = \ml{E}_{i_{1}}(u_{i_{2}}\sum_{i \in I} s^{i_{1},i_{2}}_{i_{2},i}\ml{e}_{i}) = u_{i_{2}}s^{i_{1},i_{2}}_{i_{2},i_{1}}\ml{e}_{i_{1}} \] and \[ \textstyle \ml{E}_{i_{1},i_{2}}(u) = u_{i_{2}}\ml{e}_{i_{1}} \] hence \[ \ml{E}_{i_{1}} \cdot \varphi(\ml{E}_{i_{1},i_{2}}) \cdot \ml{E}_{i_{2}} = s^{i_{1},i_{2}}_{i_{2},i_{1}}\ml{E}_{i_{1},i_{2}} \] for all $i_{1},i_{2}$. Since $\ml{E}_{i_{1}} \cdot \ml{E}_{i_{1},i_{2}} \cdot \ml{E}_{i_{2}} = \ml{E}_{i_{1},i_{2}}$, applying $\varphi$ and \labelcref{20180329-aaga-eqn-02} gives $\varphi(\ml{E}_{i_{1},i_{2}}) = s^{i_{1},i_{2}}_{i_{2},i_{1}}\ml{E}_{i_{1},i_{2}}$ for all $i_{1},i_{2}$ (and $s^{i_{1},i_{2}}_{j_{1},j_{2}} = 0$ if $(j_{1},j_{2}) \ne (i_{2},i_{1})$). We set $s_{i_{1},i_{2}} := s^{i_{1},i_{2}}_{i_{2},i_{1}}$. Applying the above argument to $\varphi^{-1}$ implies that every $s_{i_{1},i_{2}}$ is a unit. Applying $\varphi$ to $\ml{E}_{i_{1},i_{2}} \cdot \ml{E}_{i_{2},i_{3}} = \ml{E}_{i_{1},i_{3}}$ gives $s_{i_{1},i_{2}} s_{i_{2},i_{3}} = s_{i_{1},i_{3}}$ (and in particular $s_{i,i} = 1$). We choose an arbitrary $t \in I$, define $\ml{U}'' : V \to V$ by $\ml{U}'' = \sum_{i} s_{t,i}\ml{E}_{i}$; after replacing $\varphi$ by $\varphi \cdot (\ml{U}''^{-1}(-)\ml{U}'')$, we may assume \begin{align} \label{20180329-aaga-eqn-03} \varphi(\ml{E}_{i_{1},i_{2}}) = \ml{E}_{i_{1},i_{2}} \end{align} for all $i_{1},i_{2}$. Then we conclude by \Cref{20210514-10}. \end{proof}

\begin{lemma} \label{20210514-10} Let $R$ be a ring, let $I$ be an index set, set $\mc{A} := \Mat_{I \times I}^{\mr{cf}}(R)$ and let $\xi : \mc{A} \to \mc{A}$ be an $R$-algebra automorphism. For $i_{1},i_{2} \in I$, let $\ml{E}_{i_{1},i_{2}} \in \mc{A}$ denote the $(i_{1},i_{2})$th matrix unit. If $\xi(\ml{E}_{i_{1},i_{2}}) = \ml{E}_{i_{1},i_{2}}$ for all $i_{1},i_{2} \in I$, then $\xi = \id_{\mc{A}}$. \end{lemma} \begin{proof} Let $\ml{N} \in \mc{A}$ be a matrix and let $a_{i_{1},i_{2}}$ be the $(i_{1},i_{2})$th entry of $\ml{N}$. Then \[ \ml{E}_{i_{1}} \cdot \ml{N} \cdot \ml{E}_{i_{2}} = a_{i_{1},i_{2}} \ml{E}_{i_{1},i_{2}} \] where we set $\ml{E}_{i} := \ml{E}_{i,i}$; applying $\xi$ gives \[ \ml{E}_{i_{1}} \cdot \xi(\ml{N}) \cdot \ml{E}_{i_{2}} = a_{i_{1},i_{2}} \ml{E}_{i_{1},i_{2}} \] which implies that the $(i_{1},i_{2})$th entries of $\xi(\ml{N})$ and $\ml{N}$ are equal for all $i_{1},i_{2} \in I$, i.e. $\xi(\ml{N}) = \ml{N}$. \end{proof}

\begin{lemma} \label{20180329-aafx} Let $R$ be a ring, let $V$ be a free $R$-module, let $S$ be an $R$-algebra which is finitely presented as an $R$-module. The map \[ \End_{R}(V) \otimes_{R} S \to \End_{S}(V \otimes_{R} S) \] is an isomorphism. \end{lemma} \begin{proof} This follows from \cite[059K]{SP}. \end{proof}

\begin{lemma} \label{20180329-aagz} Let $\mb{A}$ be an abelian category, let $\mc{E} \in \mb{A}$ be an object, and let $q_{1},q_{2} \in \Hom_{\mb{A}}(\mc{E},\mc{E})$ and $u_{1},u_{2} \in \Aut_{\mb{A}}(\mc{E})$ be morphisms such that $u_{1}q_{1} = q_{2}u_{2}$. Then $\im q_{1} \simeq \im q_{2}$. \end{lemma}

\subsection{The center of endomorphism rings of projective modules} \label{sec07}

Given a ring $A$, it is well-known that the center of the matrix ring $\Mat_{n \times n}(A)$ consists of matrices of the form $f \id_{n}$ for some $f \in A$. In this section we prove an extension of this fact to endomorphism rings of projective modules of possibly infinite rank.

\begin{lemma} \label{20210804-12} Let $\mc{C}$ be a locally ringed site, let $\mc{E}$ be a locally projective $\mc{O}_{\mc{C}}$-module of positive rank, let $\varphi : \mc{E} \to \mc{E}$ be an $\mc{O}_{\mc{C}}$-linear endomorphism. Suppose that, for all $U \in \mc{C}$, the restriction $\varphi|_{U}$ is contained in the center of $\End_{\mc{O}_{U}}(\mc{E}|_{U})$. There exists a unique $f \in \Gamma(\mc{C},\mc{O}_{\mc{C}})$ such that $\varphi = f \id_{\mc{E}}$. \end{lemma} \begin{proof} For any $\ml{v} \in \Gamma(\mc{C},\mc{E})$, let $s_{\ml{v}} : \mc{O}_{\mc{C}} \to \mc{E}$ denote the $\mc{O}_{\mc{C}}$-linear map sending $1 \mapsto \ml{v}$. We say that $\ml{v}$ is a \DEF{unimodular} element of $\mc{E}$ if $s_{\ml{v}}$ admits an $\mc{O}_{\mc{C}}$-linear retraction $\pi_{\ml{v}} : \mc{E} \to \mc{O}_{\mc{C}}$; in this case evaluating $\varphi \circ (s_{\ml{v}} \circ \pi_{\ml{v}}) = (s_{\ml{v}} \circ \pi_{\ml{v}}) \circ \varphi$ at $\ml{v}$ implies $\varphi(\ml{v}) = \pi_{\ml{v}}(\varphi(\ml{v})) \cdot \ml{v}$, namely $\varphi(\ml{v})$ is a scalar multiple of $\ml{v}$. \par Given two unimodular elements $\ml{v}_{1},\ml{v}_{2}$ of $\mc{E}$ and retractions $\pi_{\ml{v}_{1}},\pi_{\ml{v}_{2}}$ of $s_{\ml{v}_{1}},s_{\ml{v}_{2}}$ respectively, evaluating $\varphi \circ (s_{\ml{v}_{i}} \circ \pi_{\ml{v}_{2}}) = (s_{\ml{v}_{1}} \circ \pi_{\ml{v}_{2}}) \circ \varphi$ at $\ml{v}_{2}$ implies $\varphi(\ml{v}_{1}) = \pi_{\ml{v}_{2}}(\varphi(\ml{v}_{2})) \cdot \ml{v}_{1}$; applying $\pi_{\ml{v}_{1}}$ gives $\pi_{\ml{v}_{1}}(\varphi(\ml{v}_{1})) = \pi_{\ml{v}_{2}}(\varphi(\ml{v}_{2}))$; hence this constant $\pi_{\ml{v}}(\varphi(\ml{v}))$ is independent of choice of $\ml{v}$ and retraction $\pi_{\ml{v}}$. \par After a localization of $\mc{C}$, we may assume that $\mc{E}$ is a globally a direct summand of a free module, so that we have $\mc{O}_{\mc{C}}$-linear maps $\iota : \mc{E} \to \mc{O}_{\mc{C}}^{\oplus I}$ and $\pi : \mc{O}_{\mc{C}}^{\oplus I} \to \mc{E}$ such that $\pi \circ \iota = \id_{\mc{E}}$; let $\ml{M} \in \Mat_{I \times I}^{\mr{cf}}(\Gamma(\mc{C},\mc{O}_{\mc{C}}))$ be the matrix corresponding to the $\mc{O}_{\mc{C}}$-linear map $\iota \circ \pi : \mc{O}_{\mc{C}}^{\oplus I} \to \mc{E} \to \mc{O}_{\mc{C}}^{\oplus I}$, and let $\ml{v}_{i}$ denote the $i$th column of $\ml{M}$. If $\ml{v}_{i}$ contains an entry which is a unit, then $\ml{v}_{i}$ is unimodular. The assumption that $\mc{E}$ has positive rank means that the entries of $\ml{M}$ generate the unit ideal of $\Gamma(\mc{C},\mc{O}_{\mc{C}})$: given an object $U \in \mc{C}$ for which there exists a surjection $p : \mc{E}|_{U} \to \mc{O}_{U}$, we have that $p \circ \pi$ is a surjection which remains surjective after precomposition by $\iota \circ \pi$. Let $U \in \mc{C}$ be an object; by \cite[04ES]{SP}, there exists a covering $\{U_{\lambda} \to U\}_{\lambda \in \Lambda}$ such that for each $\lambda \in \Lambda$, there exists an entry of $\ml{M}$ which is a unit of $\Gamma(U_{\lambda},\mc{O}_{\mc{C}})$; say that $\ml{v}_{i_{\lambda}}$ contains such an entry. By the above, there exists a constant $f_{\lambda} \in \Gamma(U_{\lambda},\mc{O}_{\mc{C}})$ such that $\varphi(\ml{v}_{i_{\lambda}}) = f_{\lambda} \ml{v}_{i_{\lambda}}$ in $\Gamma(U_{\lambda},\mc{E})$; moreover, we have $f_{\lambda_{1}}|_{U_{\lambda_{1}} \times_{U} U_{\lambda_{2}}} = f_{\lambda_{2}}|_{U_{\lambda_{1}} \times_{U} U_{\lambda_{2}}}$ for all $\lambda_{1},\lambda_{2} \in \Lambda$; hence there exists a unique $f \in \Gamma(U,\mc{O}_{\mc{C}})$ such that $f|_{U_{\lambda}} = f_{\lambda}$ for all $\lambda \in \Lambda$. \par Now it remains to show that, for any object $U \in \mc{C}$ and any element $\ml{w} \in \Gamma(U,\mc{E})$, there exists a covering $\{U_{\lambda} \to U\}_{\lambda \in \Lambda}$ such that for each $\lambda \in \Lambda$, the restriction $\ml{w}|_{U_{\lambda}} \in \Gamma(U_{\lambda},\mc{E})$ is an $\Gamma(U_{\lambda},\mc{O}_{\mc{C}})$-linear combination of unimodular elements of $\mc{E}|_{U_{\lambda}}$. Set $A := \Gamma(U,\mc{O}_{\mc{C}})$, let $P$ be the image of the $A$-linear map $\ml{M}|_{U} : A^{\oplus I} \to A^{\oplus I}$; since $\ml{M}$ is idempotent, we have that $P$ is a projective $A$-module. Let $\mf{p}$ be a prime of $A$; by Kaplansky \cite[3.3]{BASS-BPMAF1963}, we have that $P_{\mf{p}}$ is a free $A_{\mf{p}}$-module; we may choose finitely many elements $\ml{u}_{1}',\dotsc,\ml{u}_{n}' \in P_{\mf{p}}$ such that $\{\ml{u}_{1}',\dotsc,\ml{u}_{n}'\}$ is part of an $A_{\mf{p}}$-basis for $P_{\mf{p}}$ and such that $\ml{w}$ is an $A_{\mf{p}}$-linear combination of the $\ml{u}_{\ell}'$, say $\ml{w} = c_{1}'\ml{u}_{1}' + \dotsb + c_{n}'\ml{u}_{n}'$ with $c_{\ell}' \in A_{\mf{p}}$. We may find some $a \in A \setminus \mf{p}$ such that $c_{1}',\dotsc,c_{n}'$ lift to $c_{1}'',\dotsc,c_{n}'' \in A_{a}$ and $\ml{u}_{1}',\dotsc,\ml{u}_{n}'$ lift to $\ml{u}_{1}'',\dotsc,\ml{u}_{n}'' \in P_{a}$; we may multiply $a$ by an element of $A \setminus \mf{p}$ if necessary so that $\ml{w} = c_{1}''\ml{u}_{1}'' + \dotsb + c_{n}''\ml{u}_{n}''$ in $P_{a}$. Since each $\ml{u}_{\ell}'$ is unimodular in $P_{\mf{p}}$, each $\ml{u}_{\ell}''$ contains an entry which is not contained in $\mf{p}A_{a}$; thus we may again multiply $a$ by an element of $A \setminus \mf{p}$ so that these entries become units of $A_{a}$; then each $\ml{u}_{\ell}''$ is unimodular in $\mc{E}|_{U'}$ for any morphism $U' \to U$ such that $a|_{U'}$ is invertible in $\Gamma(U',\mc{O}_{\mc{C}})$. \par Since $\mc{C}$ is locally ringed, there exists a covering $\{U_{\lambda_{\mf{p}}} \to U\}_{\lambda_{\mf{p}} \in \Lambda_{\mf{p}}}$ such that for each $\lambda_{\mf{p}} \in \Lambda_{\mf{p}}$ either $a|_{U_{\lambda_{\mf{p}}}}$ or $(1-a)|_{U_{\lambda_{\mf{p}}}}$ is invertible in $\Gamma(U_{\lambda_{\mf{p}}},\mc{O}_{\mc{C}})$. In this way, since $\Spec A$ is quasi-compact, we obtain elements $a_{1},\dotsc,a_{n} \in A$ which generate the unit ideal of $A$ and coverings $\{U_{\lambda_{\ell}} \to U\}_{\lambda_{\ell} \in \Lambda_{\ell}}$ such that either $a_{\ell}|_{U_{\lambda_{\ell}}}$ or $(1-a_{\ell})|_{U_{\lambda_{\ell}}}$ is invertible in $\Gamma(U_{\lambda_{\ell}},\mc{O}_{\mc{C}})$; taking the fiber product of these coverings gives the result. \end{proof}

\bibliography{allbib}
\bibliographystyle{alpha}
\end{document}